\documentclass[reqno,UKenglish, 11pt]{amsart}
\usepackage[margin=3cm]{geometry}
\usepackage{amsthm} 
\usepackage{amssymb}
\usepackage{amsopn}
\usepackage{color}
\usepackage[english]{babel} 
\usepackage{hyperref}
\usepackage{youngtab}
\usepackage{amsmath, amsfonts, amssymb} 
\usepackage{graphicx}
\allowdisplaybreaks 
\usepackage[utf8]{inputenc} 
\usepackage{xfrac}    
\usepackage[draft]{todonotes}
\usepackage{tikz}
\usetikzlibrary{topaths,calc,positioning}
\usepackage{tikz-cd}
\tikzset{bluenode/.style={circle,fill=blue!50,minimum size=0.4cm,inner sep=0pt},}
  
\usepackage{subfig}

\usepackage{amssymb}

\usepackage{mathtools} 
\mathtoolsset{showonlyrefs,showmanualtags}
\numberwithin{equation}{section}
\usepackage{amsmath}
\title[Random geometric complexes and graphs]{Random geometric complexes and graphs on Riemannian manifolds in the thermodynamic limit}
\date{}
\author{Antonio Lerario}
\address{SISSA, Trieste, Italy}
\email{lerario@sissa.it}
\author{Raffaella Mulas}
\address{Max Planck Institute for Mathematics in the Sciences, Leipzig, Germany}
\email{raffaella.mulas@mis.mpg.de}

\newcommand{\bracenom}{\genfrac{\lbrace}{\rbrace}{0pt}{}}
\newcommand*\diff{\mathop{}\!\mathrm{d}}
\newcommand{\starabove}{\overset{*}}

\newcommand{\be}{\begin{equation}}
\newcommand{\ee}{\end{equation}}
\newcommand{\R}{\mathbb{R}}

\usepackage{amssymb, amsmath}
\usepackage{verbatim}
\usepackage{hyperref}

\DeclareMathOperator{\vol}{vol}
\DeclareMathOperator{\id}{id}
\DeclareMathOperator{\supp}{supp}
\DeclareMathOperator{\spec}{spec}
\DeclareMathOperator{\Sym}{Sym}
\DeclareMathOperator{\tv}{tv}
\DeclareMathOperator{\diag}{diag}

\theoremstyle{plain}
\newtheorem{theorem}{Theorem}[section]
\newtheorem{lem}[theorem]{Lemma}
\newtheorem{cor}[theorem]{Corollary}

\newtheorem{prop}[theorem]{Proposition}

\theoremstyle{definition}

\newtheorem{definition}[theorem]{Definition}

\newtheorem{ex}[theorem]{Example}

\theoremstyle{remark}
\newtheorem{rmk}[theorem]{Remark}

\usepackage{tikz}
\usetikzlibrary{positioning}
\tikzset{main node/.style={circle,fill=orange!20,draw,minimum size=1cm,inner sep=0pt},}

\makeindex
\begin{document}
	\maketitle
	\begin{abstract} We investigate some topological properties of random geometric complexes and random geometric graphs on Riemannian manifolds in the thermodynamic limit. In particular, for random geometric complexes we prove that the normalized counting measure of connected components, counted according to isotopy type, converges in probability to a deterministic measure. More generally, we also prove similar convergence results for the counting measure of types of components of each $k$--skeleton of a random geometric complex. As a consequence, in the case of the $1$--skeleton (i.e. for random geometric graphs) we show that the empirical spectral measure associated to the normalized Laplace operator converges to a deterministic measure. 
	\end{abstract}
	\section{Introduction}
	\subsection{Random geometric complexes}The subject of random geometric complexes has recently attracted a lot of attention, with a special focus on the study of expectation of topological properties of these complexes \cite{Kahle2011, Penrose, NSW, YSA, BobrowskiMukherjee, BobrowskiKahle}\footnote{This list is by no mean complete, see \cite{BobrowskiKahle} for a survey and a more complete set of references!} (e.g. number of connected components, or more generally Betti numbers). In a recent paper \cite{Antonio}, Auffinger, Lerario and Lundberg have imported methods from \cite{NazarovSodin, SarnakWigman} for the study of finer properties of these random complexes, namely the distribution of the homotopy types of the connected components of the complex. Before moving to the content of the current paper, we discuss the main ideas from \cite{Antonio} and introduce some terminology.
	
	Let $(M,g)$ be a compact, Riemannian manifold of dimension $m$. We normalize the metric $g$ in such a way that
	\be \mathrm{vol}(M)=1.\ee
	We denote by $\hat{B}(x,r)\subset M$ the Riemannian ball centered at $x$ of radius $r>0$ and we construct a \emph{random $M$--geometric complex} in the \emph{thermodynamic regime} as follows. We let $\{p_1,\ldots,p_n\}$ be a set of points independently sampled from the uniform distribution on $M$, we fix a positive number $\alpha>0$, and we consider:
		\be\label{eq:alpha}
		\mathcal{U}_n:=\bigcup_{k=1}^n\hat{B}(p_k,r)\quad \textrm{where}\quad  r:=\alpha n^{-1/m}.\ee
		The choice of such $r$ is what defines the so-called \emph{critical} or \emph{thermodynamic regime}\footnote{Quoting the Introduction from \cite{Antonio}: random geometric complexes are studied \emph{within three main phases or regimes based on the relation between density of points and radius of the neighborhoods determining the complex: the subcritical regime (or ``dust phase'') where there are many connected components with little topology, the critical regime (or ``thermodynamic regime'') where topology is the richest (and where the percolation threshold appears), and the supercritical regime where the connectivity threshold appears.
The thermodynamic regime is seen to have the most intricate topology.}}  and it's the regime where topology is the richest \cite{Antonio,Kahle}. We say that $\mathcal{U}_n$ is a \emph{random $M$--geometric complex} and the name is motivated by the fact that, for $n$ large enough, $\mathcal{U}_n$ is homotopy equivalent to its \v{C}ech complex, as we shall see in Lemma \ref{nervelemma} below.
		
		Auffinger, Lerario and Lundberg \cite{Antonio} proved that, in the case when $\vol M=1$, the normalized counting measure of connected components of such complexes, counted according to homotopy type, converges in probability to a deterministic measure. That is,
		\begin{equation}\label{equation:oldconvergence}
		    \tilde{\Theta}_n:=\frac{1}{b_0(\mathcal{U}_n)}\sum\delta_{[u]}\xrightarrow[n\rightarrow\infty]{\mathbb{P}}\tilde{\Theta},
		\end{equation}where the sum is over all connected components $u$ of $\mathcal{U}_n$, $[u]$ denotes their homotopy type and $b_0$ is the zero-th Betti number, therefore $b_0(\mathcal{U}_n)$ is the number of connected components of $\mathcal{U}_n$. In \eqref{equation:oldconvergence} the measure $\tilde{\Theta}_n$ is a \emph{random} probability measure on the countable set of all possible homotopy types of connected geometric complexes and the convergence is in probability with respect to the total variation distance (see Section \ref{section:Random measures} for more precise definitions). The support of the limiting deterministic measure $\tilde{\Theta}$ equals the set of all homotopy types for Euclidean geometric complexes of dimension $m = \dim M$. Roughly speaking, \eqref{equation:oldconvergence} tells that, for every fixed homotopy type  $[u]$ of connected geometric complexes, denoting by $\mathcal{N}_n([u])$ the random variable ``number of connected components of $\mathcal{U}_n$ which are in the homotopy equivalence class $[u]$'', there is a convergence of the random variable $\mathcal{N}_n([u])/b_0(\mathcal{U}_n)$ to a constant $c_{[u]}$ as $n\to \infty$ (the convergence is in $L^1$ and $c_{[u]}>0$ if and  only if $[u]$ contains a $\R^m$--geometric complex).  
		
		\subsection{Isotopy classes of geometric complexes}
		\begin{figure}
		\begin{center}
			\includegraphics[width=5cm]{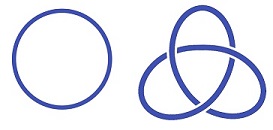}
		\end{center}
		\caption{The unknot and the trefoil knot are homotopy equivalent but they are not isotopic.}\label{Figureknots}
	\end{figure}We now move to the content of the current paper. Our first goal is to include the results of \cite{Antonio} into a more general framework which allows to make even more refined counts (e.g. according to the type of the embedding of the components, or on the structure of their skeleta, or on the property of containing a given motif\footnote{A \emph{motif} in a graph (or more generally in a complex) is a recurrent and statistically significant sub-graph or pattern.}). The first result that we prove is that (\ref{equation:oldconvergence}) still holds if we consider \emph{isotopy classes} instead of homotopy classes: intuitively, two complexes are isotopic if the vertices of one can be moved continuously to the vertices of the other without ever changing the combinatorics of the intersection of the corresponding balls (see Definition \ref{def:isotopy}). From now on we will always make the assumption that our complexes are \emph{nondegenerate}, i.e. that the boundaries of the balls defining them intersect transversely (see Definition \ref{def:nondeg}); our random geometric complexes will be nondegenerate with probability one, and the notion of isotopic nondegenerate complexes coincides with the one from differential topology.
		 In Theorem \ref{mainthm} we show that
	\begin{equation}\label{equation:newconvergence}
	    \Theta_n\xrightarrow[n\rightarrow \infty]{\mathbb{P}}\Theta,
	\end{equation}where $\Theta_n$ is defined in a similar way as $\tilde{\Theta}_n$ above, with isotopy classes instead of homotopy classes Interestingly, the limiting measure depends only on $\alpha$ on the dimension of $M$.\footnote{In the rest of the paper we will consider $\alpha$ as fixed from the very beginning and omit it from the notation; the study of the dependence of the various objects on $\alpha>0$ is  an interesting problem, on which for now we cannot say much.}. 
To appreciate the difference with the results from \cite{Antonio}: the unknot and the trefoil knot in $\R^3$ (Figure \ref{Figureknots}) are homotopy equivalent but they are not isotopic, and with positive probability there are connected $M$--geometric complexes whose embedding looks like these two knots (see Proposition \ref{propexistence} below); Theorem \ref{mainthm} is able to distinguish between them, whereas the construction from \cite{Antonio} is not.

\subsection{A cascade of measures}Theorem \ref{mainthm} contains in a sense the richest possible information on the topological structure of our geometric complexes and the convergence of many other counting measures can be deduced from it. To explain this idea, we consider the space
\be \left(\mathcal{G}/{\cong}\right):=\{\textrm{isotopy classes of connected geometric complexes}\}\ee
and we put an equivalence relation $\rho$ on $\mathcal{G}/{\cong}$ (the relation can be for example: two isotopy classes are the same if their $k$--skeleta are isomorphic, or if they contain the same number of a given motif). Then the natural map $\psi:\left(\mathcal{G}/{\cong}\right)\to \left(\mathcal{G}/{\cong}\right)/{\rho}$ defines the random pushforward measure $\psi_*\Theta_n$ on $\left(\mathcal{G}/{\cong}\right)/{\rho}$ and Theorem \ref{mainthm} implies that $\psi_*\Theta_n\to \psi_*\Theta$.

This idea can be used to produce a ``cascade'' of random relevant measures. Consider in fact the following diagram of maps and spaces:
\be 
\begin{tikzcd}
\mathcal{G}/{\cong} \arrow[r, "\varphi"] \arrow[rd, "\varphi^{(k)}"'] \arrow[rr, "\phi", bend left=49] & \mathcal{G}/{\simeq} \arrow[d] \arrow[r] & \mathcal{G}/{\sim} \\
                                                                                                       & \mathcal{G}^{(k)}/{\simeq}               &                   \end{tikzcd}\ee
where the spaces are:
\begin{align} \left(\mathcal{G}/{\simeq}\right)&:=\{\textrm{isomorphism classes of connected geometric \v{C}ech complexes}\}\\
\left(\mathcal{G}^{(k)}/{\simeq}\right)&:=\{\textrm{isomorphism classes of components of the $k$--skeleton of \v{C}ech complexes}\}\\
\left(\mathcal{G}/{\sim}\right)&:=\{\textrm{homotopy classes of connected geometric complexes}\}\end{align}
and the maps are the natural ``forgetful'' maps. For example, the map $\varphi$ takes the isotopy class of a nondegenerate complex and associates to it its homotopy class; the map $\varphi^{(k)}$ associates to it the isomorphism class of its $k$--skeleton (it is well defined since isotopic complexs have isomorphic \v{C}ech complexes).
Then for all the pushforward measures defined by these maps we have convergence in probability with respect to the total variation distance (see Section \ref{section:Random measures}) and as $n\to \infty$
\be \phi_*\Theta_n\to\phi_*\Theta,\quad  \varphi_*\Theta_n\to\varphi_*\Theta\quad \textrm{and}\quad \varphi^{(k)}_*\Theta_n\to\varphi^{(k)}_*\Theta.\ee

\subsection{Random Geometric Graphs}

Of special interest is the case of random geometric graphs: vertices of a random $M$--geometric graph $\Gamma_n$ are the points $\{p_1,\ldots,p_n\}$ and we put an edge between $p_i$ and $p_j$ if and only if $i\neq j$ and $\hat{B}(p_i, r)\cap\hat{B}(p_j, r)\neq \emptyset$. Using the above language, a random $M$--geometric graph is the $1$--skeleton of the \v{C}ech complex associated to the complex $\mathcal{U}_n$.

	To every random $M$--geometric graph $\Gamma_n$ we can associate the measure:
\be \label{eq:graphsintro}\varphi^{(1)}_{*}\Theta_n= \frac{1}{b_0(\Gamma_n)}\sum\delta_{\gamma}\to \varphi^{(1)}_*\Theta,\ee
where the sum is over all connected components of $\Gamma_n$ and $\gamma$ denotes their isomorphism class (as graphs). There is an interesting fact regarding the random variable $b_0(\Gamma_n)$ appearing in \eqref{eq:graphsintro}: it is the same random variable as $b_0(\mathcal{U}_n)$ (the number of components of the random graph and of the random complex are the same), and in \cite{Antonio} it is proved that there exists a constant $\beta$ (depending on the parameter $\alpha$ in \eqref{eq:alpha}) such that:
\be
		\frac{b_0(\Gamma_n)}{n}=\frac{b_0(\mathcal{U}_n)}{n}\xrightarrow{L^1}\beta.
\ee

The existence of this limit also follows from \cite{GoTrTs}, where the authors establish a limit law in the thermodynamic regime for Betti numbers
of random geometric complexes built over possibly inhomogeneous Poisson point processes in Euclidean space, including the case when the point process is supported on a submanifold.
Moreover, we note that for a related model of random graphs (the Poisson model on $\R^m$, see Section \ref{sec:poissonintro} below) Penrose \cite{Penrose} has proved that there exists a constant $\beta$ (depending on the parameter $\alpha$ in \eqref{eq:alpha}) such that the normalized component count converges to a constant in $L^2$. In fact, as we will see below, related to our $M$--geometric model there is a way to construct a corresponding $\R^m$--geometric model, which is in a sense the rescaled limit of the Riemannian one, and the limit constants for the two models are the same.

In fact the limit measure $\varphi^{(1)}_{*}\Theta$ also comes from the rescaled Euclidean limit and it is supported on connected $\R^m$--geometric graphs. For a given $m$, the set of such graphs is not easy to describe, but in the case $m=1$ they can be characterized by a result of Roberts \cite{unitgraphs2}, and from this result we can deduce a description of the support of the limit measure in \eqref{eq:graphsintro} (see Corollary \ref{cor:limit1} and Section \ref{section:Geometric graphs} for more details).
\begin{rmk}[Related work on random geometric graphs]
The general theory of random graphs has been founded in 1959 by Erd\"os and Rényi, who proposed a model of random graph $G(n,p)$ where the number of vertices is fixed to be $n$ and each pair of distinct vertices is joined by an edge with probability $p$, independently of other edges \cite{erdos1,erdos2,erdos3,erdos4,referee3}. Later on, other models have been proposed in the literature \cite{referee2}, as for instance the Barabási–Albert scale-free network model \cite{referee2-ref2} and the Watts-Strogatz small-world network model \cite{referee2-ref53}. For general references on random graphs, the reader is referred to \cite{referee3-ref10,Chung,chung-random,referee3-ref16,referee3-ref18,referee3-ref23,referee3-ref30,referee3-ref34,referee3-ref40}.	Here we focus on the random geometric graph model and we refer to \cite{referee2-ref23,referee2-ref42,referee2-ref50} for more literature on this topic. Applications of random geometric graphs can be found, for instance, in the contexts of wireless networks, epidemic spreading, city growth, power grids, protein-protein interaction networks \cite{referee2}. \end{rmk}
\subsection{The spectrum of a random geometric graph}
When talking about a graph, a natural associated object to look at is its normalized Laplace operator, see Section \ref{section:Laplacian of a graph and its spectrum}. It is known that the spectrum of the (symmetric) normalized Laplace operator for graphs encodes important information about the graphs \cite{Chung}. For example, it tells us how many connected components a graph has; it tells whether a graph is bipartite and whether it is complete; it tells us how difficult it is to partition the vertex set of a graph into two disjoint sets $V_1$ and $V_2$ such that the number of edges between $V_1$ and $V_2$ is as small as possible and such that the \emph{volume} of both $V_1$ and $V_2$, i.e. the sum of the degrees of their vertices, is as big as possible. Therefore, the normalized Laplace operator gives a partition of graphs into families and \emph{isospectral graphs} share important common features. Since, furthermore, the computation of the eigenvalues can be performed with tools from linear algebra, such operator is a very powerful and used tool in graph theory and data analytics.

In the context of random $M$--geometric graphs, the convergence of the counting measure in \eqref{eq:graphsintro} can be used to deduce the existence of a limit measure for the spectrum of the normalized Laplace operator for random geometric graphs. More specifically, we define the \emph{empirical spectral measure} of a graph as the normalized counting measure of eigenvalues of the normalized Laplace operator and we prove that there exists a deterministic measure $\mu$ on the real line such that (Theorem \ref{teomu})
\begin{equation}\label{equation:mu}
\mu_{\Gamma_n}:=\frac{1}{n}\sum_{i=1}^{n}\delta_{\lambda_i(\Gamma_n)}\overset{*}{\underset{n\rightarrow\infty}{\rightharpoonup}}\mu.
\end{equation}Here, $\lambda_1(\Gamma_n),\ldots,\lambda_n(\Gamma_n)$ are the eigenvalues of the normalized Laplace operator of $\Gamma_n$ and the convergence in \eqref{equation:mu} means that for every continuous function $f:[0,2]\to \R$ we have:
\be \lim_{n\to \infty}\mathbb{E}\int_{[0,2]}f d\mu_{\Gamma_n}=\int_{[0,2]} f d\mu.\ee
The measure $\mu$ in \eqref{equation:mu} is far from trivial and we don't have yet a clear understanding of it: we know it is supported on the interval $[0,2]$, but for example it is not absolutely continuous with respect to Lebesgue measure (in fact $\mu(\{0\})=\beta>0$).

\begin{rmk}Interestingly, \cite{referee3} studies the convergence of $\mu_{\Gamma_n}$ as $n\rightarrow\infty$ in the case where $\Gamma_n$ is a $G(n,p)$ random graphs and the eigenvalues are the ones of the non-normalized Laplacian or the ones of the adjacency matrix. In particular, it is shown that in such context, under suitable conditions, $\mu_{\Gamma_n}$ converges to the semi-circle law if associated to the adjacency matrix and it converges to the free convolution of the standard normal distribution if associated to the non-normalized Laplacian.\end{rmk}
\begin{rmk}\label{rmk:JJ}In \cite{JJ}, Gu, Jost, Liu and Stadler introduce the notion of \emph{spectral class} of a family of graphs. Given a Radon measure $\rho $ on $[0,2]$ and a sequence $(\Gamma_n)_{n\in \mathbb{N}}$ of graphs with $\#(V(\Gamma_n))=n$, they say that this sequence belongs to the spectral class $\rho$ if $\mu_{\Gamma_n}\starabove \rightharpoonup \rho$ as $n\to \infty$. We can interpret \eqref{equation:mu} as saying that our family of random geometric graphs $(\Gamma_n)_n$ belongs to the spectral class $\mu$ (in a probabilistic sense).
		\end{rmk}
		
		\begin{rmk}[Related work on spectral theory]Similarly to the spectrum of the normalized Laplace operator, also the spectra of the non-normalized Laplacian matrix (defined in Section \ref{section:Laplacian of a graph and its spectrum}) and the one of the adjacency matrix have been widely studied. We refer the reader to \cite{Chung,referee3-ref42} for general references on spectral graph theory. We refer to \cite{referee3-ref9,Hypergraphs} for applications of spectral graph theory in chemistry and we refer to \cite{referee3-ref24,referee3-ref25,referee3-ref26,referee3-ref38,referee3-ref39,referee3-ref43,referee3-ref44} for applications in theoretical physics and quantum mechanics. For references on spectral graph theory of (not necessarily geometric) random graphs, we refer to \cite{chung-random,referee3,referee2,referee4}. In \cite{referee2}, in particular, the eigenvalues of the adjacency matrix for random gometric graphs are studied using numerical and statistical methods. Remarkably, it is shown that random geometrix graphs are statistically very similar to the other random graph models we have mentioned above: Erdős-Rényi random graphs, Barabási–Albert scale-free networks, Watts-Strogatz small-world networks. On the other hand, in \cite{referee4}, it is shown that symmetric structures abundantly occur in random geometric graphs, while the same doesn't hold for the other random graph models. Our main results on spectral graph theory for random geometric graphs, Theorem \ref{teomu} and Proposition \ref{prop:last} below, follow the same general idea as \cite{referee2} and \cite{referee4}, in the sense that we are interested on the limiting spectrum of large random geometric graphs. The main difference is that \cite{referee2} is focused on the adjacency matrix, \cite{referee4} gives a focus on the non-normalized Laplacian and we focus on the normalized Laplacian. Therefore the final implications differ very much.
\end{rmk}
\subsection{The Euclidean Poisson model}\label{sec:poissonintro} As we already observed, in \cite{Antonio}, the proof of (\ref{equation:oldconvergence}) is based on a \emph{rescaling limit} idea. Namely, one can fix $R>0$ and a point $p\in M$, and study the limit structure of the random complexes inside the ball $\hat{B}(p,Rn^{-1/d})$. The random geometric complex obtained as $n\rightarrow\infty$ can then be described as follows. Let $P:=\{p_1,p_2,\ldots\}$ be a set of points sampled from the standard spatial Poisson distribution in $\mathbb{R}^m$. For $\alpha>0$, let
		\begin{equation*}
		\mathcal{P}:=\bigcup_{p\in P}B(p,\alpha)
		\end{equation*}and let
		\begin{equation*}
		\mathcal{P}_R:=\{\text{connected components of $\mathcal{P}$ entirely contained in the interior of }B(0,R)\}.
		\end{equation*}For the random complex $\mathcal{P}_R$, one can define completely analogue measures, where now the parameter is $R>0$, and all the above discussion applies also to this model (this is discussed throughout the paper). 

\subsection*{Structure of the paper}This paper is structured as follows. In Section \ref{section:Geometric complexes} we discuss (deterministic) $M$--geometric complexes and, in particular, we define and see some properties of the set $\mathcal{G}/{\cong}$ of isotopy classes of connected, nondegenerate $M$--geometric complexes. In Section \ref{section: Random geometric complexes} we discuss random $M$--geometric complexes; in Section \ref{section:Random measures} we prove (\ref{equation:newconvergence}). Morevorer, in Section \ref{section:Geometric graphs} we define and see some properties of geometric graphs; in Section \ref{section:Laplacian of a graph and its spectrum} we recall the definition of the normalized Laplace operator $\hat{L}$ for graphs and we prove some properties of the spectral measure in the case of geometric graphs. Finally, in Section \ref{section:Random geometric graphs} we prove (\ref{equation:mu}).
\subsection*{Acknowledgements}
We are grateful to Bernd Sturmfels, because without him this paper would not exist. We are grateful to Fabio Cavalletti, J\"urgen Jost, Matthew Kahle, Erik Lundberg, Leo Mathis and Michele Stecconi for helpful comments, discussions and forbidden graphs. We are grateful to the anonymous referees for the constructive comments. 

\section{Geometric complexes}\label{section:Geometric complexes}
Throughout this paper we fix a Riemannian manifold $(M,g)$ of dimension $m$. 
	\begin{definition}[$M$--geometric complex and its skeleta]Let $p_1,\ldots,p_n$ be points in $M$ and fix $r\geq 0$. We define a \emph{$M$--geometric complex} as
		\begin{align*}
		\mathcal{U}(\{p_1,\ldots,p_n\},r):&=\bigcup_{k=1}^n\hat{B}(p_k,r)\\
		&=\{x\in M: d_M(x,\{p_1,\ldots,p_n\})\leq r\}.
		\end{align*}For $\mathcal{U}:=\mathcal{U}(\{p_1,\ldots,p_n\},r)$, we also let
		\begin{align*}
		\check{C}(\mathcal{U})&:=\check{C}(\{p_1,\ldots,p_n\},r)\\&:=\text{ nerve of the cover }\{\hat{B}(p_k,r)\}_{k=1}^n
		\end{align*}and we let
		\begin{align*}
		\check{C}^{(k)}(\mathcal{U})&:=\check{C}^{(k)}(\{p_1,\ldots,p_n\},r)\\
		&:=k-\text{skeleton of }\check{C}(\{p_1,\ldots,p_n\},r).
		\end{align*}In particular, we call $\check{C}^{(1)}(\{p_1,\ldots,p_n\},r)$ a \emph{$M$--geometric graph}. 
	\end{definition}
		\begin{rmk}\label{remark:smooth} In order to avoid unnecessary complications, in the sequel we will always assume that the injectivity radius\footnote{Recall that the injectivity radius $\textrm{inj}_p(M)$ of $M$ at one point $p$ is defined to be the largest radius of a ball in the tangent space $T_pM$ on which the exponential map $\textrm{exp}_p:T_{p}M\to M$ is a diffeomorphism and the injectivity radius of $M$ is defined as the infimum of the injectivity radii at all points: 
		\be \textrm{inj}(M)=\inf_{p\in M}\textrm{inj}_p(M).
		\ee}
		$\textrm{inj}(M)$ of $M$ is strictly positive (which is true if $M$ is compact or if $M=\R^m$ with the flat metric) and that \be \label{eq:inj(M)}0<r\leq\textrm{inj}(M).\ee 
		This requirement ensures that for every point $p\in M$ the set 
		\be \partial\hat{B}(p,r)=\{x\in M:d(x,p_k)=r\}\ee
		is smooth (in fact it is the image of the sphere of radius $r$ in the tangent space at $p$ under the exponential map, which is a diffeomorphism on $B_{T_pM}(0, \textrm{inj}(M))$).
		Observe also that for $r\leq \textrm{inj}(M)$ the ball $\widehat{B}(p,r)$ is contractible, but not necessarily geodesically convex. 
	\end{rmk}
	\begin{definition}\label{def:nondeg}
		We say that $\mathcal{U}(\{p_1,\ldots,p_n\},r)$ is \emph{nondegenerate} if for each $J=\{j_1,\ldots,j_l\}\in\bracenom{n}{l}$ the intersection \be\label{eq:transv} \bigcap_{k=1}^l\hat{B}(p_{j_k},r)\quad \textrm{is transversal}.\ee
	\end{definition}
The next result is classical and relates the homotopy of a geometric complex to the one of its associated \v{C}ech complex.
	\begin{lem}[Nerve Lemma]\label{nervelemma}
		If $M$ is compact, there exists $\rho>0$ such that, for each $r\leq \rho$,
		\begin{equation*}
		\mathcal{U}(\{p_1,\ldots,p_n\},r)\sim \check{C}(\{p_1,\ldots,p_n\},r),
		\end{equation*}i.e. they are homotopy equivalent. 
	\end{lem}
	\begin{proof}For the proof in this setting, see \cite[Lemma 6.1]{Antonio}.
	\end{proof}
	\begin{definition}[Isotopy classes of connected geometric complexes]\label{def:isotopy}Let $p_1,\ldots,p_n$ and $q_1,\ldots,q_n$ be points in $M$ and let $r_0,r_1\geq 0$ such that 
		\begin{equation*}
		\mathcal{U}_0:=\mathcal{U}(\{p_1,\ldots,p_n\},r_0)\quad\textrm{and}\quad
		\mathcal{U}_1:=\mathcal{U}(\{q_1,\ldots,q_n\},r_1)
		\end{equation*}are nondegenerate $M$--geometric complexes. We say that $\mathcal{U}_0$ and $\mathcal{U}_1$ are \emph{(rigidly) isotopic} and we write $\mathcal{U}_0\cong\mathcal{U}_1$ if there exists an isotopy of diffeomorphisms $\varphi_t:M\rightarrow M$ with $\varphi_0=\id_M$ and a continuous function $r(t)>0$, for $t\in[0,1]$, such that:
		\begin{itemize}
			\item For each $t\in[0,1]$, $\mathcal{U}(\{\varphi_t(p_1),\ldots,\varphi_t(p_n)\},r(t))$ is nondegenerate,
			\item $r(0)=r_0$,
			\item $r(1)=r_1$ and
			\item $\varphi_1(p_1)=q_1,\ldots,\varphi_1(p_n)=q_n$.
		\end{itemize}
	\end{definition}
	\begin{rmk}[Isotopy classes and discriminants]\label{rmkdiscriminant}The definition of two complexes being (rigidly) isotopic is very reminiscent of the notion of rigid isotopy from algebraic geometry, where the ``regular'' deformations are those which do not intersect some discriminant. We can make this analogy more precise. For every $n\in\mathbb{N}$ consider the smooth manifold
	\be\label{eq:H} H_{n}:=\underbrace{M\times \cdots\times M}_{\textrm{$n$ many times}}\times (0, \textrm{inj}(M)),\ee
	together with the \emph{discriminant}
	\be\label{eq:Sigma} \Sigma_{n}:=\{(p_1, \ldots, p_n, r)\,|\,\textrm{ $\mathcal{U}(\{p_1, \ldots, p_n\}, r)$ is degenerate}\}.
	\ee
	The set $\Sigma_{n}$ is closed since its complement $R_{n}$  is defined by the (finitely many) transversality conditions \eqref{eq:transv}. Adopting this point of view, isotopy classes of nondegenerate $M$--geometric complexes built using $n$ many balls are labeled by the connected components of $R_{n}:=H_{n}\backslash \Sigma_{n}$ (the complement of a discriminant).
	In fact, given a nondegenerate complex $\mathcal{U}(\{p_1, \ldots, p_n\}, r)$, then $(p_1, \ldots, p_n, r)\in R_{n}$ (because it is nondegenerate) and viceversa every point in $R_{n}$ correspond to a nondegenerate complex. Moreover, a nondegenerate isotopy of two nondegenerate complexes defines a curve between the corresponding points  of $R_{n}$, and this curve is entirely contained in $R_{n}$; the two corresponding parameters must therefore lie in the same connected component of $R_{n}$; viceversa, because for an open set of a manifold connected and path connected are equivalent, every two points in the same component $R_{n}$ can be joined by an arc all contained in $R_{n}$ and give therefore rise to isotopic complexes.
	\end{rmk}
	\begin{definition}
	        We define the set
	        \begin{equation*}
	            \mathcal{G}(M):=\{\text{connected, nondegenerate $M$--geometric complexes}\}
	        \end{equation*}and use the notation $\mathcal{G}:=\mathcal{G}(M)$ when $M$ is given. We also let
	        \begin{equation*}
		    \left(\mathcal{G}/{\cong}\right):=\{\text{isotopy classes of connected, nondegenerate $M$--geometric complexes}\}.
		\end{equation*}
	\end{definition}
	\begin{rmk}
		Observe that, by definition, each class $[\mathcal{U}]\in\mathcal{G}/{\cong}$ keeps also the information on the way $\mathcal{U}$ is embedded in $M$. In particular, it might be that to two nondegenerate $M$--geometric complexes $\mathcal{U}_1$ and $\mathcal{U}_2$ there correspond isomorphic \v{C}ech complexes $\check{C}(\mathcal{U}_1)\simeq \check{C}(\mathcal{U}_2)$, but at the same time the complexes $\mathcal{U}_1$ and $\mathcal{U}_2$ themselves are not rigidly isotopic (see Figure \ref{fig:complexes}).
	\end{rmk}

	\begin{rmk}For each $M$ of dimension $m$,
		$\mathcal{G}(M)/{\cong}\supset\mathcal{G}(\mathbb{R}^m)/{\cong}$.
	\end{rmk}
	\begin{theorem}\label{teocountable}
		$\mathcal{G}/{\cong}$ is a countable set.
	\end{theorem}
	\begin{proof}
		We first partition $\mathcal{G}/{\cong}$ into countably many sets. For every $n\in \mathbb{N}$ we consider the set:
		\be
		\left(\mathcal{G}_{n}/{\cong}\right):=\{\textrm{classes of complexes in $\mathcal{G}/{\cong}$ which are built using $n$ many balls}\},\ee
		and we need to prove that this set is countable.
		We have already seen (Remark \ref{rmkdiscriminant}) that isotopy classes of nondegenerate complexes which are built using $n$ many balls are in one-to-one correspondence with the connected components of $R_{n}=H_{n}\backslash \Sigma_{n}$ (these sets are defined by \eqref{eq:H} and \eqref{eq:Sigma}). The function ``number of connected components of a nondegenerate complex'' is constant on each component of $R_{n}$ and consequently the number of isotopy classes of \emph{connected} and nondegenerate complexes (i.e. the cardinality of $\mathcal{G}_{n}/{\cong}$) is smaller than the number of components of $R_{n}$:
		\be \#\bigl(\mathcal{G}_{n}/{\cong}\bigr)\leq \#\bigl(\{\textrm{connected components of $R_{n}$}\}\bigr).\ee
		
		We are therefore reduced to prove that $R_{n}$ has at most countably many components. To this end we write $R_{n}$ as the disjoint union of its components (each of which is an open set in $H_{n}$):
		\be \label{eq:Ca}R_{n}=\bigsqcup_{\alpha \in A}C_{\alpha}.\ee
		We cover now the manifold $H_{n}$ with countably many manifold charts $\{(V_j, \varphi_j)\}_{j\in \mathbb{N}}$ with $\varphi_j:V_j\stackrel{\sim}{\longrightarrow} \R^m$. For every $j\in \mathbb{N}$ we consider also the decomposition of the open set $R_{n}\cap V_j$ into its connected components:
		\be R_{n}\cap V_j=\bigsqcup_{\beta \in B_j}C_{\beta, j}.\ee
		Since each $C_\alpha$ from \eqref{eq:Ca} is the union of elements of the form 
		\be C_\alpha=\bigcup_{j\in N_\alpha,\, \beta \in B_j}C_{\beta, j}\ee
		with the index set $j$ running over the countable set $N_\alpha\subset \mathbb{N}$, it is therefore enough to prove that for every $j\in \mathbb{N}$ the index set $B_{j}$ is countable, i.e. that the number of connected components of $R_{n}\cap V_j$ is countable.
		Observe now that, since $\varphi_j$ is a diffeomorphism between $V_j$ and $\R^m$, then the number of connected components of $R_{n}\cap V_j$ is the same of the number of connected components of $\varphi_j(R_{n}\cap V_j),$ which is an open subset of $\R^m$. Since in each component of $\varphi_j(R_{n}\cap V_j)$ we can pick a point with rational coordinates, it follows that the number of such components is countable, and this concludes the proof.
		
		\end{proof}
		\begin{rmk}Observe that the key point of the proof of Theorem \ref{teocountable} is showing that the number of connected components of an open set in a differentiable manifold is countable.
		\end{rmk}
	\begin{definition}
		We let 
		\begin{equation*}
		    \left(\mathcal{G}/{\sim}\right):=\{\text{homotopy classes of connected, nondegenerate $M$--geometric complexes}\}
		\end{equation*}
		and we let
		\begin{equation*}
		\left(\mathcal{G}^{(k)}/{\simeq}\right):=\{\text{isomorphism classes of $k$--skeleta of connected $M$--geometric complexes}\}.
		\end{equation*}In particular,
		\begin{equation*}
		\mathcal{G}^{(1)}/{\simeq}=\{\text{isomorphism classes of connected $M$--geometric graphs}\}.
		\end{equation*}
	\end{definition}
	\begin{rmk}\label{rmkpallettine}In order to appreciate the difference between the classes in $\left(\mathcal{G}/{\cong}\right)$, $\left(\mathcal{G}/{\sim}\right)$ and $\left(\mathcal{G}^{(k)}/{\simeq}\right)$, we look at Figure \ref{fig:complexes}. Here, we have three $\R^3$--geometric complexes, given by the union of balls in $\R^3$, that have three different shapes. All three complexes are homotopy equivalent to each other and they are all pairwise not isotopic. Moreover, while the first one and the second one have isomorphic $1$--skeleta, the $1$--skeleton of the third complex is not isomorphic to the other ones.
	\end{rmk}
	\begin{figure}
		\begin{center}
				\includegraphics[width=9cm]{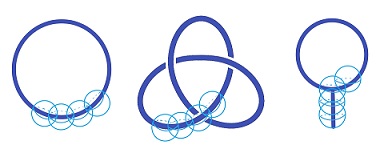}

	\caption{The three $\R^3$--geometric complexes described in Remark \ref{rmkpallettine}. From left to right, we call them $\mathcal{U}_1, \mathcal{U}_2$ and $ \mathcal{U}_3$ and we assume the number of balls and their combinatorics needed to to describe the first two are the same. Then $[\mathcal{U}_1]=[\mathcal{U}_2]=[\mathcal{U}_3]$ in $\mathcal{G}/{\sim}$, because they are all homotopy equivalent, $[\mathcal{U}_1]=[\mathcal{U}_2]\neq[\mathcal{U}_3]$ in $\mathcal{G}/{\simeq}$ (the first two give rise to same \v{C}ech complexes, which is however different from the last one) and $[\mathcal{U}_1]\neq[\mathcal{U}_2]\neq[\mathcal{U}_3]\neq [\mathcal{U}_1]$ in $\mathcal{G}/{\cong}$ (they are all pairwise non-isotopic).}\label{fig:complexes}
	\end{center}
	\end{figure}
	\begin{rmk}\label{rmk:forgetful}
		There are natural forgetful maps
		\begin{align*}
		\phi:\mathcal{G}/{\cong}\longrightarrow&\mathcal{G}/{\sim}\\
		[\mathcal{U}]\longmapsto&[\mathcal{U}]
		\end{align*}and
		\begin{align*}
		\varphi^{(k)}:\mathcal{G}/{\cong}\longrightarrow&\mathcal{G}^{(k)}/{\simeq}\\
		[\mathcal{U}]\longmapsto&[\check{C}^{(k)}(\mathcal{U})].
		\end{align*}
	\end{rmk}
	\begin{definition}[Component counting function]
		Given a nondegenerate geometric complex $\mathcal{U}\subset M$, a topological subspace $Y\subset M$ and a class $w\in\mathcal{G}/{\cong}$, we define
		\begin{equation*}
		\mathcal{N}(\mathcal{U};w):=\#\bigl(\text{components of $\mathcal{U}$ of type $w$}\bigr);
		\end{equation*}
		\begin{equation*}
		\mathcal{N}(\mathcal{U},Y;w):=\#\bigl(\text{components of $\mathcal{U}$ of type $w$ entirely contained in the interior of $Y$}\bigr);
		\end{equation*}
		\begin{equation*}
		\mathcal{N}^*(\mathcal{U},Y;w):=\#\bigl(\text{components of $\mathcal{U}$ of type $w$ intersecting $Y$}\bigr).
		\end{equation*}
	\end{definition}
	In the paper \cite{NazarovSodin} Nazarov and Sodin have introduced a powerful tool (the ``integral geometry sandwitch'') for localizing the count of the number of components of the zero set of random waves in a Riemannian manifold. This tool has been used by Sarnak and Wigman \cite{SarnakWigman} for the study of distribution of components type of the zero set of random waves on a Riemannian manifold, and it has been adapted to geometric complexes in \cite{Antonio}. We recall here this tool, stated in the language of this paper. 
	\begin{theorem}[Analogue of the Integral Geometry Sandwich]\label{teosandwich}The following two estimates are true:
	\begin{enumerate}
	\item \emph{(The local case)} Let $\mathcal{U}$ be a generic geometric complex in $\mathbb{R}^m$ and fix $w\in \mathcal{G}(\R^m)/{\cong}$. Then for $0<r<R$
	\begin{equation*}
	\int_{B_{R-r}}\frac{\mathcal{N}(\mathcal{U},B(x,r);w)}{\vol(B_r)}\diff x\leq \mathcal{N}(\mathcal{U},B_R;w)\leq\int_{B_{R+r}}\frac{\mathcal{N}^*(\mathcal{U},B(x,r);w)}{\vol (B_r)}\diff x.
	\end{equation*}
	\item \emph{(The global case)} Let $\mathcal{U}$ be a generic geometric complex in a compact Riemannian manifold $M$ and fix $w\in \mathcal{G}(M)/{\cong}$. Then for every $\varepsilon>0$ there exists $\eta>0$ such that for every $r<\eta$:
	\be (1-\varepsilon)\int_{M}\frac{\mathcal{N}(\mathcal{U},B(x,r);w)}{\vol(B_r)}\diff x\leq \mathcal{N}(\mathcal{U},M;w)\leq(1+\varepsilon)\int_{M}\frac{\mathcal{N}^*(\mathcal{U},B(x,r);w)}{\vol (B_r)}\diff x.
	\ee
	\end{enumerate}
\end{theorem}
\begin{proof}The proof of both statements is exactly the same as in \cite{SarnakWigman}, after noticing that the only property needed on the counting functions is that the considered topological spaces only have finitely many components, and these components are counted according to a specific type (here are selected according to isotopy type, but we could consider instead any function that partitions the set of components and count only the components belonging to a given class).
\end{proof}
	
	\section{Random geometric complexes (thermodynamic regime)}\label{section: Random geometric complexes}
	\begin{definition}[Riemannian case]Let $M$ be a compact Riemannian manifold and consider a set of points $\{p_1,\ldots,p_n\}$ independently sampled from the uniform distribution on $M$. Fix a positive number $\alpha>0$, let $r:=\alpha n^{-1/m}$ and let
		\begin{equation*}
		\mathcal{U}_n:=\mathcal{U}(\{p_1,\ldots,p_n\},r).
		\end{equation*}We say that $\mathcal{U}_n$ is a \emph{random $M$--geometric complex}. The choice of such $r$ is what defines the so-called \emph{critical} or \emph{thermodynamic regime}.
	\end{definition}
	\begin{definition}[Euclidean Poisson case]
		Let $P:=\{p_1,p_2,\ldots\}$ be a set of points sampled from the standard spatial Poisson distribution in $\mathbb{R}^m$. For $\alpha>0$, let
		\begin{equation*}
		\mathcal{P}:=\bigcup_{p\in P}B(p,\alpha)
		\end{equation*}and, for $R>0$, let
		\begin{equation*}
		\mathcal{P}_R:=\{\text{connected components of $\mathcal{P}$ entirely contained in the interior of }B(0,R)\}.
		\end{equation*}
	\end{definition}
	\begin{rmk}
		With probability $1$, we have that $\#\bigl(P\cap B(0,R)\bigr)$ is finite. To see this, observe that
		\begin{equation*}
		\mathbb{P}\{\#\bigl(P\cap B(0,R)\bigr)\geq l\}=\sum_{k\geq l}\frac{\vol(B(0,R))}{k!}e^{-\vol(B(0,R))}
		\end{equation*}and
		\begin{equation*}
		\mathbb{P}\{\#\bigl(P\cap B(0,R)\bigr)=\infty\}\leq\mathbb{P}\{\#\bigl(P\cap B(0,R)\bigr)\geq l\}\xrightarrow{l\rightarrow\infty}0.
		\end{equation*}
	\end{rmk}
	
From now on, we will only consider nondegenerate complexes, without further mentioning this assumption. This is not reductive, since our random complexes are nondegenerate with probability one.	
	\section{Random measures}\label{section:Random measures}
	We fix the following notation. Given a set $A$ with a fixed sigma algebra (omitted from the notation), we denote by:
			\begin{equation*}
			\mathcal{M}(A):=\{\text{measures on }A\}
\qquad\text{and}\qquad
		\mathcal{M}^1(A):=\{\text{probability measures on }A\}.
		\end{equation*}

	\begin{definition}
			Let $\mathcal{U}\subset M$ be a finite geometric complex and let
		\begin{equation*}
		\mathcal{U}=\mathcal{U}^1\sqcup\ldots\sqcup\mathcal{U}^{b_0(\mathcal{U})}
		\end{equation*}be its decomposition into connected components.
		We define $\Theta_\mathcal{U}\in\mathcal{M}^1(\mathcal{G}/{\cong})$ as
		\begin{equation*}
		\Theta_\mathcal{U}:=\frac{1}{b_0(\mathcal{U})}\sum_{j=1}^{b_0(\mathcal{U})}\delta_{[\mathcal{U}^j]}.
		\end{equation*}
		Observe that the measure $\Theta_{\mathcal{U}}$ just defined is a probability measure.
		We also endow $\mathcal{M}^1(\mathcal{G}/{\cong})$ with the \emph{total variation distance}:
		\begin{equation*}
		d_{\tv}(\Theta_1,\Theta_2):=\sup_{A\subset\mathcal{G}/{\cong}}|\Theta_1(A)-\Theta_2(A)|.
		\end{equation*}When $\mathcal{U}$ is a random geometric complex, $\Theta_{\mathcal{U}}$ is a random variable with values in the metric space $(\mathcal{M}^1(\mathcal{G}/{\cong}),d_{\tv})$. In this context, recall the notion of \emph{convergence in probability}:
		\begin{equation*}
		\Theta_n\xrightarrow{\mathbb{P}}\Theta \iff \forall\varepsilon, \lim_{n\rightarrow \infty}\mathbb{P}(d_{\tv}(\Theta_n,\Theta)\geq\varepsilon)=0.
		\end{equation*}Using the previous notation, we set
		\begin{equation*}
		\Theta_n:=\Theta_{\mathcal{U}_n} \qquad\text{and}\qquad\Theta_R:=\Theta_{\mathcal{P}_R}.
		\end{equation*}
	\end{definition}
	\begin{theorem}\label{mainthm}
		There exists a probability measure $\Theta\in\mathcal{M}^1(\mathcal{G}/{\cong})$ such that
		\begin{enumerate}
			\item $\Theta_n\xrightarrow[n\rightarrow \infty]{\mathbb{P}}\Theta\qquad$ and $\qquad\Theta_R\xrightarrow[R\rightarrow \infty]{\mathbb{P}}\Theta$
			\item $\supp(\Theta)=\mathcal{G}(\R^m)/{\cong}$.
		\end{enumerate}
	\end{theorem}
	We shall see the proof of Theorem \ref{mainthm} in Section \ref{sectionmainproof}. As a first corollary we recover the results from \cite{Antonio}.
	\begin{cor}[Theorem 1.1 and Theorem 1.3 from \cite{Antonio}]
		Consider the forgetful map \begin{align*}
		\phi:\mathcal{G}/{\cong}\longrightarrow&\,\mathcal{G}/{\sim}\\
		[\mathcal{U}]\longmapsto&[\mathcal{U}].
		\end{align*}We have that
		\begin{equation*}
		\phi_*\Theta_n\xrightarrow{\mathbb{P}}\phi_*\Theta\qquad\text{and}\qquad\phi_*\Theta_R\xrightarrow{\mathbb{P}}\phi_*\Theta.
		\end{equation*}Also, $\supp\phi_*\Theta=\mathcal{G}(\R^m)/{\sim}$.
	\end{cor}
	As a second corollary we see that, because Theorem \ref{mainthm} keeps track of fine properties of the geometric complex, we can use other forgetful maps and obtain information on the limit distribution of the components type of each $k$--skeleton.
	\begin{cor}\label{corsupppihik}
		For each $k\in\mathbb{N}$, consider the forgetful map from Remark \ref{rmk:forgetful}:
		\begin{align*}
		\varphi^{(k)}:\mathcal{G}/{\cong}\longrightarrow&\,\mathcal{G}^{(k)}/{\simeq}\\
		[\mathcal{U}]\longmapsto&[\check{C}^{(k)}(\mathcal{U})].
		\end{align*}We have that
		\begin{equation*}
		\varphi^{(k)}_*\Theta_n\xrightarrow{\mathbb{P}}\varphi^{(k)}_*\Theta\qquad\text{and}\qquad\varphi^{(k)}_*\Theta_R\xrightarrow{\mathbb{P}}\varphi^{(k)}_*\Theta.
		\end{equation*}Also, $\supp\varphi^{(k)}_*\Theta=\mathcal{G}^{(k)}(\mathbb{R}^m)/{\simeq}$.
	\end{cor}
\begin{prop}[Existence of all isotopy types]\label{propexistence}
		Let $\mathcal{U}$ be a nondegenerate geometric complex in $\mathbb{R}^m$ and let $\alpha>0$. Let $(\mathcal{U}_n)_n$ be a sequence of random $M$--geometric complexes constructed using $\alpha$. There exist $n_0$, $R$, $c>0$ (depending on $\Gamma$ and $\alpha$ but not on $M$) such that for every $p\in M$ and for every $n\geq n_0$:
		\begin{equation*}
		\mathbb{P}\{\mathcal{U}_n\cap\hat{B}(p,Rn^{-1/m})\cong\mathcal{U}\}>c.
		\end{equation*}
	\end{prop}
	\begin{proof}The proof is similar to the proof of \cite[Proposition 1.2]{Antonio}. Here we sketch this proof for the sake of completeness, pointing out what is the main difference with \cite{Antonio}. 
	
	Assume that $\mathcal{U}\subset B(0, R')$ is constructed using balls of radius $r=1$:
	\be \mathcal{U}=\bigcap_{k=1}^\ell B(y_k, 1),\ee 
	set $R=\alpha R'$ and
	consider the sequence of maps:
	\be\psi_n:\hat{B}(p, Rn^{-1/m})\stackrel{\mathrm{exp}_p^{-1}}{\longrightarrow}B_{T_pM}(0, Rn^{-1/m})\stackrel{\mathrm{dilation}}{\longrightarrow}B_{T_pM}(0, R')\simeq B(0, R')\ee
	 For $n>$ large enough the map $\psi_n$ becomes a diffeomorphism and we denote by $\varphi_n$ its inverse.  \cite[Proposition 6.2]{Antonio} implies that there exist $\epsilon_0>0$ and $n_0>0$ such that if $\|\tilde{y}_k-y_k\|\leq \epsilon_0$ for every $k=1, \ldots, \ell$, then for $n\geq n_0$ the two complexes $\bigcup_{k=1}^\ell \hat{B}(\varphi_n(\tilde y_k), \alpha n^{-1/m})$ and $\mathcal{U}$ are isomorphic. In fact, because $\mathcal{U}$ is nondegenerate, possibly choosing $\epsilon_0>0$ even smaller, we can make sure that the complex $\bigcup_{k=1}^\ell \hat{B}(\varphi_n(\tilde y_k), \alpha n^{-1/m})$ belong to the same rigid isotopy class of $\mathcal{U}$, because this is an open condition.
	 
	One then proceeds considering the event:
\be E_{n}=\left\{\exists I_\ell\in\genfrac{\{ }{\}}{0pt}{}{n}{\ell}\,:\, \forall j\in I_\ell \quad p_{j}\in \psi_n^{-1}(B(y_j, \epsilon)), \text{ and} \quad \forall j\notin I_\ell \quad p_j\in \hat{B}(p, (R+\alpha)n^{-1/m})^c\right\}.\ee

Observe that:
\be E_n\implies \mathcal{U}_n\cap\hat{B}(p,Rn^{-1/m})\cong\mathcal{U},\ee
and in particular, in order to get the conclusion, it is enough to estimate from below the probability of $E_n$. This is done in the last lines of the proof of \cite[Proposition 1.2]{Antonio}.

%
\end{proof}

		\begin{ex}
			Let $C_{137}$ be the cycle with $137$ vertices and $137$ edges. By Proposition \ref{propexistence} there exist $n_0$, $R$ and $c>0$ (depending on $C_{137}$ and $\alpha$ but not on $M$) such that for every $p\in M$ and for every $n\geq n_0$:
			\begin{equation*}
			\mathbb{P}\{\Gamma_n\cap\hat{B}(p,Rn^{-1/m})\cong C_{137}\}>c.
			\end{equation*}
		\end{ex}
	\subsection{Proof of Theorem \ref{mainthm}}\label{sectionmainproof}
We split the statement of Theorem \ref{mainthm} into two parts. The first part, Theorem \ref{theo:main1}, states that the random measure $\Theta_R$ converges in probability to a deterministic measure $\Theta\in\mathcal{M}^1(\mathcal{G}/{\cong})$ supported on the set $\mathcal{G}(\R^m)/{\cong}$. The second part, Theorem \ref{theo:main2}, states that also the random measure $\Theta_n$ converges in probability to $\Theta$.

\subsubsection{The local model}
	\begin{prop}\label{prop2.1}
		For every $w\in \mathcal{G}(\R^m)/{\cong}$ there exists a constant $c_w>0$ such that the random variable
		\begin{equation*}
		c_{R,w}:=\frac{\mathcal{N}(\mathcal{P}_R,w)}{\vol(B(0,R))}
		\end{equation*}converges to $c_w$ in $L^1$ and almost surely as $R\rightarrow\infty$.
	\end{prop}

	\begin{proof}[Proof of Proposition \ref{prop2.1}]
		Following the proof of \cite[Proposition 2.1]{Antonio} with $\mathcal{N}(\mathcal{P}_R,\tau,w)$ instead of $\mathcal{N}(\mathcal{P}_R,\tau,\gamma)$ and with the application of Theorem \ref{teosandwich} instead of \cite[Theorem 6.6]{Antonio}, one proves that there exists a constant $c_w$ such that
	\begin{equation*}
	\frac{\mathcal{N}(\mathcal{P}_R,B(0,R);w)}{\vol(B(0,R))}\longrightarrow c_w.
	\end{equation*}Since $\mathcal{P}_R\subset B(0,R)$, this implies that 
	\begin{equation*}
	\frac{\mathcal{N}(\mathcal{P}_R,w)}{\vol(B(0,R))}\longrightarrow c_w.
	\end{equation*}We now have to prove that $c_w>0$. Since $w\in \mathcal{G}(\R^m)/{\cong}$, given $\mathcal{U}\in w$ there exist $\beta>0$ and $y_1,\ldots,y_n\in\mathbb{R}^m$ such that
	\begin{equation*}
	    \mathcal{U}\cong\mathcal{U}(\{y_1,\ldots,y_n\},\beta).
	\end{equation*}Let $R_1$ be such that $\mathcal{U}\subset B(0,R_1)$ and choose $r$ such that $r\beta=\alpha$, where $\alpha$ is the constant that we used for constructing the random complex $\mathcal{P}$. We can then rescale $\mathcal{U}$ in $B(0,r\cdot R_1)$ so that it is constructed on radius $\alpha$. Now, since $\mathcal{U}$ is nondegenerate, there exists $\varepsilon>0$ such that, if for every $i$ we have that $\|y_i-y'_i\|<\varepsilon$, then the complex 
	\begin{equation*}
	    \mathcal{U}'(\{y_1',\ldots,y_n'\},\alpha )
	\end{equation*}
is isotopic to $\mathcal{U}$.
	 Take $K_R=k\cdot\mathrm{vol}(B(0, R))$ disjoint balls $\{B(y_j,r\cdot R_1)\}_{j=1,\ldots,K_R}$ inside $B(0,R)$ with $c>0$. Then
	\begin{equation*}
	\frac{\mathcal{N}(\mathcal{P}_R,B(0,R);w)}{\vol(B(0,R))}\geq \frac{1}{\vol(B(0,R))}\cdot\sum_{j=1}^{K_R}\mathcal{N}(\mathcal{P}_R,B(y_j,r\cdot R_1);w).
	\end{equation*}Therefore
	\begin{align*}
	\mathbb{E}\Biggl(\frac{\mathcal{N}(\mathcal{P}_R,B(0,R);w)}{\vol(B(0,R))}\Biggr)&\geq \mathbb{E}\Biggr(\frac{1}{\vol(B(0,R))}\cdot\sum_{j=1}^{K_R}\mathcal{N}(\mathcal{P}_R,B(y_j,r\cdot R_1);w)\Biggr)\\
	&=\frac{K_R}{\vol(B(0,R))}\cdot\mathbb{E}\Biggl(\mathcal{N}(\mathcal{P}_R,B(y,r\cdot R_1);w)\Biggr),
	\end{align*}since the random variables $\mathcal{N}(\mathcal{P}_R,B(y_j,r\cdot R_1);w)$ are identically distributed by the fact that the balls are disjoint. Now,
	\begin{equation*}
	\frac{K_R}{\vol(B(0,R))}\geq 	\frac{k\cdot \mathrm{vol}(B(0, R))}{\vol(B(0,R))}=k>0
\end{equation*} 
and  
	\begin{equation*}
	\mathbb{E}\Biggl(\mathcal{N}(\mathcal{P}_R,B(y,r\cdot R_1);w)\Biggr)>0.
	\end{equation*}Therefore, $c_w>0$.
	\end{proof}
			Proposition \ref{prop2.1} allows us to deduce the following theorem.
	\begin{theorem}\label{theo:main1}There exists $\Theta\in\mathcal{M}^1(\mathcal{G}/{\cong})$ such that
	\begin{equation*}
	   \Theta_R\xrightarrow[R\rightarrow \infty]{\mathbb{P}}\Theta.
	\end{equation*}
		 Also, $\supp(\Theta)=\mathcal{G}(\R^m)/{\cong}$.
\end{theorem}
\begin{proof}The proof is the same as the one of \cite[Theorem 1.3]{Antonio}, replacing the homotopy type counting function with the isotopy type one. 
\end{proof}
	\subsubsection{Riemannian case}
	\begin{theorem}\label{teo3.1}
		Let $p\in M$. For every $\delta>0$ and for $R>0$ sufficiently big there exists $n_0$ such that for every $w\in \mathcal{G}(\R^m)/{\cong}$ and for $n\geq n_0$:
		\begin{equation*}
		\mathbb{P}\{\mathcal{N}(\mathcal{P}_R,B(0,R);w)=\mathcal{N}(\mathcal{U}_n,\hat{B}(p,Rn^{-1/m});w)\}\geq 1-\delta.
		\end{equation*}
	\end{theorem}
		\begin{proof}The proof is the same as the one of \cite[Theorem 3.1]{Antonio}, with the following differences:
	\begin{itemize}
		\item In point (3), instead of considering the homotopy equivalence between the unions of the balls, we consider the isotopy equivalences between their $k$--skeleta. This is allowed because, at the end of the proof of point (3), it is proved that the combinatorics of the covers are the same.
		\item After assuming point (1), point (2) and the modified point (3), we can say that the $k$--skeleta of the two unions of balls are isotopic and also the unions of all the components entirely contained in $B(0,R)$ (respectively $\hat{B}(p,Rn^{-1/m})$) have isotopic $k$--skeleta. In particular, the number of components of a given isotopy class $w$ is the same for both sets with probability at least $1-\delta$.
	\end{itemize}
		\end{proof}
		\begin{cor}\label{cor3.2}
		For each $w\in \mathcal{G}(\R^m)/{\cong}$, $\alpha>0$, $x\in M$ and $\varepsilon>0$, we have
		\begin{equation*}
		    \lim_{R\rightarrow\infty}\limsup_{n\rightarrow\infty}\mathbb{P}\biggl\{\biggl|\frac{\mathcal{N}(\mathcal{U}_n,B(x,Rn^{-1/d});w)}{\vol (B(0,R))}-c_w\biggr|>\varepsilon\biggr\}=0.
		\end{equation*}
				\end{cor}
					\begin{proof}
		The  proof  is  the  same  as  the  one  of  \cite[Corollary 3.2]{Antonio},  with  the application of Theorem \ref{teo3.1} and Proposition \ref{prop2.1} instead of \cite[Theorem 3.1]{Antonio} and \cite[Proposition 2.1]{Antonio}.
		\end{proof}
		\begin{theorem}\label{teo4.1}
		  For every $w\in \mathcal{G}(\R^m)/{\cong}$, the random variable
		  \begin{equation*}
		      c_{n,w}:=\frac{\mathcal{N}(\mathcal{U}_n;w)}{n}
		  \end{equation*}converges in $L^1$ to $c_w\cdot\vol(M)$, where $c_w$ is the constant appearing in Proposition \ref{prop2.1}. In particular, this implies that the random variable
		  \begin{equation*}
		      c_{n}:=\frac{\mathcal{N}(\mathcal{U}_n)}{n},
		  \end{equation*}i.e. when we consider all components with no restriction on their type, converges in $L^1$ to $c$, where $c=\sum_{w\in \mathcal{G}(\R^m)/{\cong}}c_w>0$.
		\end{theorem}
		\begin{proof}
		The proof is the same as the one of \cite[Theorem 4.1]{Antonio}, with the application of Theorem \ref{teosandwich} instead of \cite[Theorem 6.7]{Antonio} and the application of Corollary \ref{cor3.2} instead of \cite[Corollary 3.2]{Antonio}. 
		\end{proof}
		\begin{theorem}\label{theo:main2}The measure $\Theta\in\mathcal{M}^1(\mathcal{G}/{\cong})$ appearing in Theorem \ref{theo:main1} is such that 
		\begin{equation*}
		   \Theta_n\xrightarrow[n\rightarrow \infty]{\mathbb{P}}\Theta.
		\end{equation*}
		\end{theorem}
		\begin{proof}
The proof is the same as the one of \cite[Theorem 1.1]{Antonio}, with the following differences:
	\begin{itemize}
	\item We apply Theorem \ref{teo4.1} instead of \cite[Theorem 4.1]{Antonio};
	\item We use $w\in \mathcal{G}(\R^m)/{\cong}$ instead of $\gamma\in \mathcal{G}$, $\Theta_n$ instead of $\hat{\mu}_n$ and $\mathcal{G}(M)/{\cong}$ instead of $\hat{\mathcal{G}}$.
	\end{itemize}
	\end{proof}
	\section{Geometric graphs}\label{section:Geometric graphs}
	We specialize the previous discussion to the case $k=1$, and consider
	\begin{equation*}
	\mathcal{G}^{(1)}=\{\text{isomorphism classes of connected, nondegenerate $M$--geometric graphs}\}.
	\end{equation*}
	\begin{rmk}\label{rem:open}The set of $M$--geometric graphs defined using closed balls equals the set of $M$--geometric graphs defined using open balls. To see this, assume that a geometric graph $\Gamma=(V(\Gamma),E(\Gamma))$ is defined using closed balls of radius $r$. Then, for each pair of distinct vertices $(p_i,p_j)$,
		\begin{equation*}
		(p_i,p_j)\in E(\Gamma) \iff d(p_i,p_j)\leq r.
		\end{equation*}Now, choose $\varepsilon\geq 0$ small enough that, for each pair of distinct vertices $(p_i,p_j)$,
		\begin{equation*}
		(p_i,p_j)\in E(\Gamma) \iff d(p_i,p_j)< r+\varepsilon.
		\end{equation*}Therefore, $\Gamma$ can be constructed as a $M$--geometric graph using open balls of radius $r+\varepsilon$. The inverse implication is analogous. 
	\end{rmk}
	
	In the case $M=\mathbb{R}^m$, the problem of describing the set $\mathcal{G}^{(1)}(\mathbb{R}^m)$ is equivalent to asking which graphs are realizable as $\mathbb{R}^m$--geometric graphs in a given dimension $m$. There is a vast literature about this problem and, commonly, geometric graphs realizable in dimension $m$ are called \emph{$m$--sphere graphs} while the minimal dimension $m$ such that a given graph is a $m$--sphere graph is called its \emph{sphericity}. In \cite{sphericity1} it is proved that every graph has finite sphericity; in \cite{sphericity2} the authors prove that the problem of deciding, given a graph $\Gamma$, whether $\Gamma$ is a $m$--sphere is NP-hard for all $m>1$. We can also observe that, for each $m>0$, there are graphs that are not $m$--sphere graphs. To see this, consider the \emph{kissing number} $k(m)$ in dimension $m$, defined as the number of non-overlapping unit spheres that can be arranged such that they each touch a common unit sphere. Consider the star graph with a central vertex connected to $n$ external vertices, where $n>k(m)$. In order to have a realization of dimension $m$ of this graph, we need a central sphere that touches $n$ spheres which do not touch each other. Since $n>k(m)$, this is not possible.
		\begin{figure}
		\begin{center}
			\includegraphics[width=7cm]{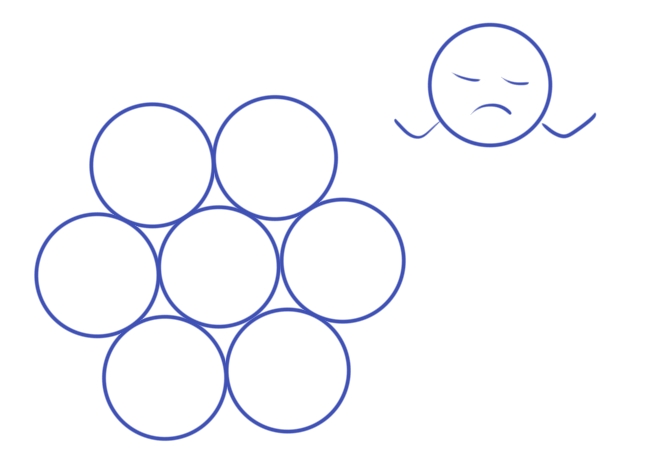}
		\end{center}
		\caption{$7$ kissing spheres in dimension $2$. The eighth sphere doesn't know where to go.}\label{Figurekissing}
	\end{figure}
	\begin{ex}
		In dimension $2$, the kissing number is $6$, as shown in Figure \ref{Figurekissing}. Therefore, any star graph $S_n$ on $n+1$ vertices, with $n>6$, is not realizable in dimension $2$ as a sphere graph.
	\end{ex}

	In the particular case of $m=1$, $1$--sphere graphs are called \emph{indifference graphs}, \emph{unit interval graphs} and there are many characterizations of such graphs \cite{unitgraphs1, unitgraphs2, unitgraphs3, unitgraphs4, unitgraphs5, unitgraphs6}. A classical characterization is due to Roberts and Wegner \cite{unitgraphs2,unitgraphs5,unitgraphs3} and it characterizes unit interval graphs by the absence of certain forbidden subgraphs, this is recalled in Theorem \ref{lekkerkerker}.
	\begin{theorem}[Roberts and Wegner]\label{lekkerkerker}
		A graph is a unit interval graph, i.e. it's an element of $\mathcal{G}^{(1)}(\mathbb{R})$, if and only if it does not contain any cycle of length at least four and any of the graphs shown in Figure \ref{fig:proi} as induced subgraph.
	\end{theorem} 
	As a consequence, we get the following corollary.
	\begin{cor}\label{cor:limit1}The support of the measure $\varphi^{(1)}_*\Theta$ defined in Corollary \ref{corsupppihik} is given by all graphs that do not contain any cycle of length at least four and any of the graphs shown in Figure \ref{fig:proi} as induced subgraph.
	\end{cor}
		\begin{figure}
    \centering
    \subfloat{{
    	\begin{tikzpicture}
			\node[bluenode] (1) {};
			\node[bluenode] (2) [below left = 0.6 cm and 0.3 cm of 1]  {};
			\node[bluenode] (3) [below right = 0.6 cm and 0.3 cm of 1] {};
			\node[bluenode] (4) [below right = 0.6 cm and 0.6 cm of 3]  {};
			\node[bluenode] (5) [below left = 0.6 cm and 0.6 cm of 2] {};
			\node[bluenode] (6) [above = 0.6 cm of 1] {};
			
			\path[draw,thick]
			(1) edge node {} (2)
			(1) edge node {} (3)
			(2) edge node {} (3)
			(4) edge node {} (3)
			(2) edge node {} (5)
			(1) edge node {} (6);
			
			\end{tikzpicture}
    }}
    \qquad
    \subfloat{{
    	\begin{tikzpicture}
			\node[bluenode] (1) {};
			\node[bluenode] (2) [below left = 1 cm and 1 cm of 1]  {};
			\node[bluenode] (3) [below right = 1 cm and 1 cm of 1] {};
	 	\node[bluenode] (4) [above = of 1] {};
			
			\path[draw,thick]
			(1) edge node {} (2)
			(1) edge node {} (4)
			(1) edge node {} (3);
			
			\end{tikzpicture}
    }}\qquad
    \subfloat{{
    	\begin{tikzpicture}
			\node[bluenode] (1) {};
			\node[bluenode] (2) [below left = 1 cm and 0.5 cm of 1]  {};
			\node[bluenode] (3) [below right = 1 cm and 0.5 cm of 1] {};
	    	\node[bluenode] (4) [below left = 1 cm and 0.5 cm of 2] {};
	    	\node[bluenode] (5) [below right = 1 cm and 0.5 cm of 2] {};
	    	\node[bluenode] (6) [below right = 1 cm and 0.5 cm of 3] {};
			
			\path[draw,thick]
			(1) edge node {} (2)
			(3) edge node {} (2)
			(1) edge node {} (3)
			(3) edge node {} (6)
			(2) edge node {} (4)
			(2) edge node {} (5)
			(3) edge node {} (5)
			(4) edge node {} (5)
			(5) edge node {} (6);
			
			\end{tikzpicture}
    }}
    \caption{Together with the cycles of length at least four, these are the non-admissible induced subgraphs for unit interval graphs on the line.}
    \label{fig:proi}
\end{figure}
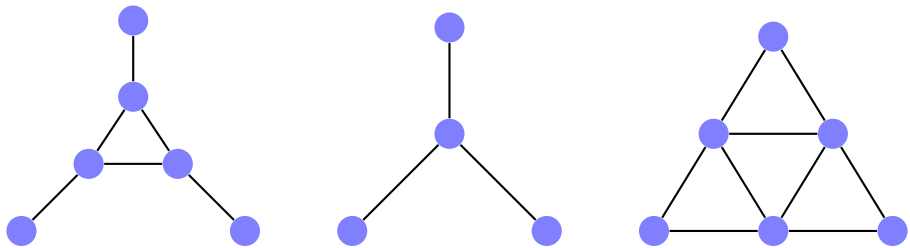
	\section{Normalized Laplacian of a graph and its spectrum}\label{section:Laplacian of a graph and its spectrum}

We now fix a graph $\Gamma$ on $n$ vertices $v_1,\ldots,v_n$ and we recall the definition of the (symmetric) normalized Laplace matrix, together with other common matrices associated to graphs. We shall then define the \emph{spectrum} and the \emph{spectral measure} associated to these matrices, and show some properties.
	\begin{definition}[Matrices associated to a graph]
		Let $A$ be the \emph{adjacency matrix} of $\Gamma$; let $D:=\diag(\deg v_1,\ldots,\deg v_n)$ be the \emph{degree matrix}; let $L:=D-A$ be the \emph{non-normalized Laplacian matrix} and let $\hat{L}:=I_n-D^{-1/2}AD^{-1/2}$ be the \emph{symmetric normalized Laplacian matrix}.
	\end{definition}
\begin{definition}[Spectrum of a matrix]
Given $Q\in \Sym(n,\mathbb{R})$ let $\spec(Q)$ be the spectrum of $Q$, i.e. the collection of its eigenvalues repeated with multiplicity,
\begin{equation*}
		\lambda_1(Q)\leq\ldots\leq \lambda_n(Q).
\end{equation*}
 We define the \emph{empirical spectral measure of} $Q$ as
\begin{equation*}
\mu_Q:=\frac{1}{n}\sum_{i=1}^n\delta_{\lambda_i(Q)}.
\end{equation*}
\end{definition}
\begin{definition}[Spectrum of a graph]\label{defspectrum}
		We define \emph{spectrum} of $\Gamma$, $\spec(\Gamma)$, as the spectrum of $\hat{L}$ and we write it as
		\begin{equation*}
		\lambda_1(\Gamma)\leq\ldots\leq \lambda_n(\Gamma).
		\end{equation*}
		We also define
		\begin{equation*}
		s(\Gamma):=\sum_{i=1}^{n}\delta_{\lambda_i(\Gamma)}
		\end{equation*}and the \emph{spectral measure of} $\Gamma$ as
		\begin{equation*}
		\mu_\Gamma:=\mu_{\hat{L}}=\frac{1}{n}\sum_{i=1}^n\delta_{\lambda_i(\Gamma)}.
		\end{equation*}
	\end{definition}
	Recall that, for every $i=1,\ldots,n$, $\lambda_i(\Gamma)\in[0,2]$ \cite[Equation (1.1) and Lemma 1.7]{Chung}. In particular, this implies that $s(\Gamma)\in\mathcal{M}([0,2])$ and $\mu_\Gamma\in\mathcal{M}^1([0,2])$.
	\begin{theorem}\label{teoo1edges}
	
			Let $(\Gamma_{1,n})_n$ and $(\Gamma_{2,n})_n$ be two sequences of graphs such that, for every $n$, $(\Gamma_{1,n})_n$ and $(\Gamma_{2,n})_n$ are two graphs on $n$ nodes that differ at most by $c$ edges. Denote by $\mu_{1,n}$ and $\mu_{2,n}$ the spectral measures associated to one of the matrices $A$, $D$, $L$, $\hat{L}$. Then
		\begin{equation*}
		\mu_{1,n}-\mu_{2,n}\starabove \rightharpoonup 0,
		\end{equation*} where $\starabove \rightharpoonup$ denotes the weak star convergence, i.e. for each $f\in C^{0}_c(\mathbb{R},\mathbb{R})$
			\begin{equation*}
			\bigg| \int_\mathbb{R}f\diff\mu_{1,n}-\int_\mathbb{R}f\diff\mu_{2,n}\bigg|\rightarrow 0.
			\end{equation*}
	\end{theorem}
	We shall prove Theorem \ref{teoo1edges} in Section \ref{sectionweakconvergence}.
		\begin{rmk}
			In the case of $\hat{L}$, we have convergence in total variation distance for ``connected sum'' of complete graphs, but not for paths, as we shall see in Section \ref{sectionstrongconvergence}.
		\end{rmk}
		\begin{rmk} Theorem 2.8 from \cite{JJ} tells that, if two families $(\Gamma_{1,n})_n$ and  $(\Gamma_{2, n})_n$ differ by at most $c$ edges \emph{and} their corresponding spectral measures have weak limits, then they belong to the same spectral class (the notion of spectral class of a family of graph is recalled in Remark \ref{rmk:JJ}). In this sense our previous Theorem \ref{teoo1edges} can be considered as an analogue of \cite[Theorem 2.8]{JJ}: the difference of the spectral measure of two families of graphs $(\Gamma_{1,n})_n$ and  $(\Gamma_{2, n})_n$ differing by at most a finite number $c$ of edges, goes to zero weakly (without the assumption that the corresponding spectral measures have weak limits).
		\end{rmk}
		
		\subsection{Proof of Theorem \ref{teoo1edges}}\label{sectionweakconvergence}
		\subsubsection{Preliminaries}
		Given $Q\in \Sym(n,\mathbb{R})$, we define the \emph{$1$--Shatten norm of} $Q$ as
		\begin{equation*}
		\|Q\|_{S^1}:=\sum_{i=1}^n|\lambda_i(Q)|.
		\end{equation*}The \emph{Weilandt-Hoffman inequality} \cite[Exercise 1.3.6]{Tao} holds:
		\begin{equation}\label{Weilandt-Hoffman inequality}
		\sum_{i=1}^n|\lambda_i(Q_1)-\lambda_i(Q_2)|\leq \|Q_1-Q_2\|_{S^1}.
		\end{equation}
		\begin{prop}\label{prop1proof}
			Let $Q_1,Q_2\in \Sym(n,\mathbb{R})$ such that 
			\begin{equation*}
			\|Q_1-Q_2\|_{S^1}\leq C.
			\end{equation*}Then, for each $f\in C^0_c(\mathbb{R},\mathbb{R})$ and for each $\varepsilon>0$, there exists $\delta>0$ such that
			\begin{equation*}
			\biggl|\int_{\mathbb{R}}f\diff\mu_{Q_1}-\int_{\mathbb{R}}f\diff\mu_{Q_2}\biggr|\leq\varepsilon+\frac{2\sup |f|}{\delta n}\cdot C.
			\end{equation*}
		\end{prop}
		\begin{proof}
			Denote by $\{\lambda_i^{(1)}\}_{i=1}^n$ and $\{\lambda_i^{(2)}\}_{i=1}^n$ the eigenvalues of $Q_1$ and $Q_2$ respectively. Then
			\begin{equation*}
			\int_{\mathbb{R}}f\diff\mu_{Q_1}-\int_{\mathbb{R}}f\diff\mu_{Q_2}=\frac{1}{n}\sum_{i=1}^{n}f(\lambda_i^{(1)})-f(\lambda_i^{(2)}),
			\end{equation*}therefore
			\begin{equation*}
			\biggl|\int_{\mathbb{R}}f\diff\mu_{Q_1}-\int_{\mathbb{R}}f\diff\mu_{Q_2}\biggr|\leq\frac{1}{n}\sum_{i=1}^{n}\bigl|f(\lambda_i^{(1)})-f(\lambda_i^{(2)})\bigr|.
			\end{equation*}Now, since $f\in C^0_c(\mathbb{R},\mathbb{R})$, $f$ is uniformly continuous and given $\varepsilon>0$ there exists $\delta=\delta(f)$ such that 
			\begin{equation*}
			|\lambda_1-\lambda_2|\leq\delta \qquad \Longrightarrow \qquad |f(\lambda_1)-f(\lambda_2)|\leq\varepsilon.
			\end{equation*}Therefore, since by Equation \ref{Weilandt-Hoffman inequality} and by hypothesis we have that
			\begin{equation*}
			\sum_{i=1}^{n}|\lambda_i^{(1)}-\lambda_i^{(2)}|\leq \|Q_1-Q_2\|_{S^1}\leq C,
			\end{equation*}it follows that 
			\begin{equation*}
			|\{|\lambda_i^{(1)}-\lambda_i^{(2)}|>\delta\}|\leq\frac{C}{\delta}.
			\end{equation*}Therefore,
			\begin{align*}
			\biggl|\int_{\mathbb{R}}f\diff\mu_{Q_1}-\int_{\mathbb{R}}f\diff\mu_{Q_2}\biggr|&\leq\frac{1}{n}\sum_{i=1}^{n}\bigl|f(\lambda_i^{(1)})-f(\lambda_i^{(2)})\bigr|\\
			&= \frac{1}{n}\sum_{|\lambda_i^{(1)}-\lambda_i^{(2)}|<\delta}\bigl|f(\lambda_i^{(1)})-f(\lambda_i^{(2)})\bigr|+\frac{1}{n}\sum_{|\lambda_i^{(1)}-\lambda_i^{(2)}|\geq\delta}\bigl|f(\lambda_i^{(1)})-f(\lambda_i^{(2)})\bigr|\\
			&\leq \frac{1}{n}\sum_{|\lambda_i^{(1)}-\lambda_i^{(2)}|<\delta}\varepsilon+\frac{1}{n}\sum_{|\lambda_i^{(1)}-\lambda_i^{(2)}|\geq\delta}\bigl|f(\lambda_i^{(1)})\bigr|+\bigl|f(\lambda_i^{(2)})\bigr|\\
			&\leq\frac{1}{n}\cdot\varepsilon\cdot\bigl|\{|\lambda_i^{(1)}-\lambda_i^{(2)}|<\delta\}\bigr|+\frac{1}{n}\cdot 2\sup |f|\cdot\bigl|\{|\lambda_i^{(1)}-\lambda_i^{(2)}|	\geq\delta\}\bigr|\\
			&\leq\frac{1}{n}\cdot\varepsilon\cdot n + \frac{1}{n}\cdot\frac{2\sup |f|}{\delta}\cdot C\\
			&\leq \varepsilon+\frac{2\sup |f|}{\delta n}\cdot C.
			\end{align*}
		\end{proof}
		\subsubsection{Applications to graphs}
		\begin{lem}\label{lemci}Let $\Gamma_1$, $\Gamma_2$ be two graphs with $V(\Gamma_1)=V(\Gamma_2)$ that differ by at most $C$--many edges. Then,
			\begin{align*}
			&\|A_1-A_2\|_{S^1}\leq 4C\\
			&\|D_1-D_2\|_{S^1}\leq 4C^2\\
			&\|L_1-L_2\|_{S^1}\leq 4C^2\\
			&\|\hat{L}_1-\hat{L}_2\|_{S^1}\leq 2C\cdot\sqrt{2}\cdot\sqrt{n-1}.
			\end{align*}
		\end{lem} 
		\begin{proof}Observe that any of the matrices
			\begin{equation*}
			\Delta_1:=A_1-A_2,\qquad\Delta_2:=D_1-D_2,\qquad\Delta_3:=L_1-L_2
			\end{equation*}consists of all zeros except for at most $4C$ entries, all of which entries are bounded by a constant (it is $1$ for $\Delta_1$ and $C$ for $\Delta_2$ and $\Delta_3$). Therefore, each $\Delta_i\in\Sym(n,\mathbb{R})$ for $i=1,2,3$ has rank at most $4C$ and all its eigenvalues are zero, except for at most $4C$ of them. 			It follows that, for $i=1,2,3$,
			\begin{align*}
			\|\Delta_i\|_{S^1}&=\sum_{j=n-2C+1}^n|\lambda_j(\Delta_i)|\\
			&=\langle(|\lambda_{n-2C+1}|,\ldots,|\lambda_n|),(1,\ldots,1) \rangle\\
			&\leq \sum_{j=n-2C+1}^n\bigl(\lambda_j(\Delta_i)^2\bigr)^{1/2}\sqrt{4C}\\
			&=2C^{1/2}\|\Delta_i\|_F\\
			&=2C^{1/2}\bigl(\sum_{j,k}(\Delta_i)_{jk}^2\bigr)^{1/2}\\
			&=2C^{1/2}(4C\tilde{C}^2)^{1/2}\\
			&\leq 4C\tilde{C}_i,
			\end{align*}where
			\begin{equation*}
			    \tilde{C}_i=\begin{cases}1 & \textrm{if }i=1\\ C & \textrm{if }i=2,3.\end{cases}
			\end{equation*}

			Similarly, $\Delta_4:=\hat{L}_1-\hat{L}_2$ consists of all zeros except for at most $2C(n-1)$ entries, all of which entries are bounded by $1$, and it has rank at most $4C$. Therefore, 
				\begin{align*}
			\|\Delta_4\|_{S^1}&=\sum_{j=1}^n|\lambda_j(\Delta_4)|\\
			&\leq \sum_{j=1}^n\bigl(\lambda_j(\Delta_4)^2\bigr)^{1/2}\cdot(4C)^{1/2}\\
			&=\bigl(\sum_{j,k}(\Delta_4)_{jk}^2\bigr)^{1/2}\cdot2C^{1/2}\\
			&\leq \bigl(\sum_{1}^{2C(n-1)}1\bigr)^{1/2}\cdot2C^{1/2}\\
			&=2C\cdot\sqrt{2}\cdot\sqrt{n-1}.
			\end{align*}
		\end{proof}As a corollary, we can prove Theorem \ref{teoo1edges}.
		\begin{proof}[Proof of Theorem \ref{teoo1edges}]We prove that, for each $f\in C^{0}_c(\mathbb{R},\mathbb{R})$ and for each $\varepsilon>0$,
			\begin{equation*}
			\lim_{n\rightarrow\infty}\bigg| \int_\mathbb{R}f\diff\mu_{1,n}-\int_\mathbb{R}f\diff\mu_{2,n}\bigg|\leq \varepsilon.
			\end{equation*}Let $c_1:=4C$, $c_2:=c_3:=4C^2$ and $c_4:=2C\cdot\sqrt{2}$. By Lemma \ref{lemci}, we have that 
			\begin{equation*}
			\|\Delta_i\|_{S^1}\leq c_i
			\end{equation*}for each $i=1,2,3$. By Proposition \ref{prop1proof}, there exists $\delta>0$ such that
			\begin{equation*}
			\biggl|\int_{\mathbb{R}}f\diff\mu_{1,n}-\int_{\mathbb{R}}f\diff\mu_{2,n}\biggr|\leq\varepsilon+\frac{2\sup |f|}{\delta n}\cdot c_i.
			\end{equation*}Therefore,
			\begin{equation*}
			\lim_{n\rightarrow\infty}\bigg| \int_\mathbb{R}f\diff\mu_{1,n}-\int_\mathbb{R}f\diff\mu_{2,n}\bigg|\leq \varepsilon.
			\end{equation*}Similarly, by Lemma \ref{lemci} we have that 
			\begin{equation*}
			\|\Delta_4\|_{S^1}\leq c_4\sqrt{n-1}.
			\end{equation*}By Proposition \ref{prop1proof}, there exists $\delta>0$ such that
			\begin{equation*}
			\biggl|\int_{\mathbb{R}}f\diff\mu_{1,n}-\int_{\mathbb{R}}f\diff\mu_{2,n}\biggr|\leq\varepsilon+\frac{2\sup |f|}{\delta n}\cdot c_4\sqrt{n-1}.
			\end{equation*}Therefore,
			\begin{equation*}
			\lim_{n\rightarrow\infty}\bigg| \int_\mathbb{R}f\diff\mu_{1,n}-\int_\mathbb{R}f\diff\mu_{2,n}\bigg|\leq \varepsilon.
			\end{equation*}
		\end{proof}
		\subsection{Strong convergence for complete graphs}\label{sectionstrongconvergence}
			\begin{lem}\label{lemmaKN}
				Given $N\in\mathbb{N}^+$, let $K_N$ and $K'_N$ be two complete graphs on $N$ nodes. Let $K_N\sqcup K'_N$ be their disjoint union and let $\Gamma_N:=K_N\bigcup_{c\text{ edges }} K'_N$ be their union together with $c$ edges $(v_i,v'_i)$ where $v_i\in K_N$ and $v'_i\in K'_N$, for $i=1,\ldots,c$. Let also $\mu_{K_N\sqcup K'_N}$ and $\mu_{\Gamma_N}$ be the spectral measures of these two graphs. Then,
				\begin{equation*}
				\mu_{K_N\sqcup K'_N}=\frac{1}{N}\cdot\delta_0+\biggl(\frac{N-1}{N}\biggr)\cdot\delta_{\frac{N}{N-1}} 
				\end{equation*}and
				\begin{equation*}
				\mu_{\Gamma_N}=\frac{1}{2N}\cdot\delta_0+\frac{1}{2N}\cdot\sum_{i=1}^{2c+1}\delta_{a_i}+\biggl(\frac{N-1-c}{N}\biggr)\cdot\delta_{\frac{N}{N-1}}
				\end{equation*}for some $a_i\in (0,2)$.
			\end{lem}
			\begin{rmk}In order to prove Lemma \ref{lemmaKN}, we make the following observation. It is easy to see that the spectrum of the symmetric normalized Laplacian matrix $\hat{L}=I_n-D^{-1/2}AD^{-1/2}$ equals the spectrum of the \emph{random walk normalized Laplacian matrix} $\tilde{L}:=I_n-D^{-1}A$. Moreover, for a graph with vertex set $V$, $\tilde{L}$ can be seen as an operator from the set $\{f:V\rightarrow\R\}$ to itself. We shall work on this operator for proving Lemma \ref{lemmaKN} and, for a graph $\Gamma$, we shall use the simplified notation $L^{\Gamma}$ in order to indicate the random walk normalized Laplace operator for $\Gamma$.
			\end{rmk}
			\begin{proof}[Proof of Lemma \ref{lemmaKN}]
				Since the spectrum of $K_N\sqcup K'_N$ is given by $0$ with multiplicity $2$ and $\frac{N}{N-1}$ with multiplicity $2(N-1)$, we have that
				\begin{equation*}
				\mu_{K_N\sqcup K'_N}=\frac{1}{2N}\Biggl(2\cdot\delta_0+2(N-1)\cdot\delta_{\frac{N}{N-1}}\Biggr).
				\end{equation*}In order to prove the second part of the lemma, we shall find $2(N-1-c)$ functions on $V(\Gamma_N)$ that are eigenfunctions for the normalized Laplace operator with eigenvalue $\frac{N}{N-1}$ and are orthogonal to each other. In particular, by the symmetry of $\Gamma_N$, it suffices to find $N-1-c$ such functions that are $0$ on the vertices of $K'_N$.
				
				Observe that $K_{N-1-c}$ is a subgraph of $K_N\setminus\{v_1,\ldots,v_c\}$ that has $N-1-c$ eigenfunctions $f_1,\ldots,f_{N-1-c}$ for the largest eigenvalue. These are orthogonal to each other and orthogonal to the constants, therefore
				\begin{equation*}
				0=\sum_{v\in V(K_{N-1-c})}\deg_{K_{N-1-c}}(v)f_i(v)f_j(v)=\sum_{v\in V(K_{N-1-c})}f_i(v)f_j(v)
				\end{equation*}and
				\begin{equation*}
				\sum_{v\in V(K_{N-1-c})}\deg_{K_{N-1-c}}(v)f_i(v)=\sum_{v\in V(K_{N-1-c})}f_i(v)=0
				\end{equation*}
				for each $i,j\in\{1,\ldots,N-1-c\}$. Now, for $i\in\{1,\ldots,N-1-c\}$, let $\tilde{f}_i$ be the function on $V(K_N)$ that is equal to zero on $v_1,\ldots,v_c$ and is equal to $f_i$ otherwise. Then, $\tilde{f}_1,\ldots,\tilde{f}_{N-1-c}$ are orthogonal to each other and orthogonal to the constants because
				\begin{align*}
				\sum_{v\in V(K_N)}\deg_{K_N}(v)\tilde{f}_i(v)\tilde{f}_j(v)&=\sum_{v\in V(K_{N-1-c})}(N-1)\tilde{f}_i(v)\tilde{f}_j(v)\\
				&=\sum_{v\in V(K_{N-1-c})}\tilde{f}_i(v)\tilde{f}_j(v)\\
				&=\sum_{v\in V(K_{N-1-c})}f_i(v)f_j(v)\\
				&=0
				\end{align*}and
				\begin{equation*}
				\sum_{v\in V(K_N)}\tilde{f}_i(v)=\sum_{v\in V(K_{N-1-c})}f_i(v)=0.
				\end{equation*}Since for complete graphs any function that is orthogonal to the constants is an eigenfunction for $\frac{N}{N-1}$, we have that $\tilde{f}_1,\ldots,\tilde{f}_{N-1-c}$ are (pairwise orthogonal) eigenfunctions for $\frac{N}{N-1}$. 
				
				Analogously, for $i\in\{1,\ldots,N-1-c\}$, let now $\hat{f}_i$ be the function on $V(\Gamma_N)$ that is equal to zero on $K_N'\cup \{v_1,\ldots,v_c\}$ and is equal to $f_i$ otherwise. It's then easy to see that also these functions are orthogonal to each other and orthogonal to the constants. Now, for each $i$ and for each $v\in \Gamma_N$ with $\tilde{f}_i(v)=0$, we have that $v\in K_{N-1-c}$, therefore
				\begin{align*}
				L^{\Gamma_N}\hat{f}_i(v)&=\hat{f}_i(v)-\frac{1}{\deg_{\Gamma_N}(v)}\sum_{w\sim v}\hat{f}_i(w)\\
				&=\tilde{f}_i(v)-\frac{1}{\deg_{K_N}(v)}\sum_{w\sim v\text{ in }K_N}\tilde{f}_i(w)\\
				&=L^{K_N}\tilde{f}_i(v)\\
				&=\frac{N}{N-1}\cdot\tilde{f}_i(v)\\
				&=\frac{N}{N-1}\cdot\hat{f}_i(v).
				\end{align*}
				This proves that the functions $\hat{f}_i$'s are $N-1-c$ orthogonal eigenfunctions of the Laplace Operator in $\Gamma_N$ for the eigenvalue $\frac{N}{N-1}$. Since they are all $0$ on $K'_N$, by symmetry we can also get $N-1-c$ eigenfunctions for $\frac{N}{N-1}$ on $\Gamma_N$ that are $0$ on $K_N$ and therefore are orthogonal to the first $N-1-c$ functions. This implies that the multiplicity of $\frac{N}{N-1}$ for $\Gamma_N$ is at least $2(N-1-c)$. Therefore, 
				\begin{equation*}
				\mu_{\Gamma_N}=\frac{1}{2N}\cdot\delta_0+\frac{1}{2N}\cdot\sum_{i=1}^{2c+1}\delta_{a_i}+\biggl(\frac{N-1-c}{N}\biggr)\cdot\delta_{\frac{N}{N-1}}
				\end{equation*}for some $a_i\in (0,2)$.
			\end{proof}
			\begin{cor}
				The total variation distance between the probability measures $\mu_{K_N\sqcup K'_N}$ and $\mu_{\Gamma_N}$ defined in the previous lemma is
				\begin{equation*}
				\sup_{A\subseteq[0,2] \text{ measurable }}\biggl|\mu_{K_N\sqcup K'_N}(A)-\mu_{\Gamma_N}(A)\biggr|=\frac{2c+1}{2N}.
				\end{equation*}In particular, if $c=o(N)$, the total variation distance tends to zero for $N\rightarrow\infty$.
			\end{cor}
			\begin{ex}\label{counterexamplestrong}
				The previous corollary doesn't hold in general. As a counterexample, take two copies of the path on $N$ vertices, $P_N$ and $P'_N$. Their union via one external edge can be for example the path on $2N$ vertices, and
				\begin{equation*}
				\mu_{P_{2N}}=\frac{1}{2N}\Biggl(\sum_{k=0}^{2N-1}\delta_{1-\cos\frac{\pi k}{2N-1}}\Biggr),
				\end{equation*}while
				\begin{equation*}
				\mu_{P_N\sqcup P'_N}=\frac{1}{N}\Biggl(\sum_{k=0}^{N-1}\delta_{1-\cos\frac{\pi k}{N-1}}\Biggr).
				\end{equation*}The total variation distance between these two measures does not tend to zero for $N\rightarrow\infty$.
			\end{ex}
	\section{Random geometric graphs}\label{section:Random geometric graphs}	
	Specializing Corollary \ref{corsupppihik} to the case $k=1$, we get 
	\begin{equation*}
	\varphi^{(1)}_*\Theta_n\xrightarrow{\mathbb{P}}\varphi^{(1)}_*\Theta\qquad\text{and}\qquad\varphi^{(1)}_*\Theta_R\xrightarrow{\mathbb{P}}\varphi^{(1)}_*\Theta,
	\end{equation*}with $\supp\varphi^{(1)}_*\Theta=\mathcal{G}^{(1)}(\mathbb{R}^m)$.
	\begin{definition}
		We define the random geometric graphs
		\begin{equation*}
		\Gamma_n:=\check{C}^{(1)}(\mathcal{U}_n)\qquad\text{and}\qquad\Gamma_R:=\check{C}^{(1)}(\mathcal{P}_R)
		\end{equation*}and we associate the empirical spectral measures 
		\begin{equation*}
		\mu_n:=\mu_{\Gamma_n}=\frac{1}{n}\sum_{i=1}^{n}\delta_{\lambda_i(\Gamma_n)}\qquad\text{ and }\qquad \mu_R:=\mu_{\Gamma_R}=\frac{1}{\#(V(\Gamma_R))}\sum_{i=1}^{\#(V(\Gamma_R))}\delta_{\lambda_i(\Gamma_R)}.
		\end{equation*}
	\end{definition}
 	Since the spectrum of a graph $\Gamma$ is finite, we can rewrite the two random measures $\mu_n$ and $\mu_R$ above as follows. First, the set of all possible isomorphism classes of graphs is countable and therefore there exists a sequence $\{x_\ell\}_{\ell\in \mathbb{N}}\subset [0,2]$ such that:
	\be\mu_n=\sum_{\ell=1}^\infty c_{\ell, n}\delta_{x_\ell}\quad \textrm{and}\quad \mu_R=\sum_{\ell=1}^\infty c_{\ell, R}\delta_{x_\ell},
	\ee
	where the random variables $c_{\ell, n}$ and $c_{\ell, R}$ are defined by:
	\be 
	c_{\ell, n}=\frac{1}{n}\#\{\textrm{components $\Gamma_{n,j}$ of $\Gamma_n=\sqcup\Gamma_{n, j}$ such that $x_\ell$ belongs to the spectrum of $\Gamma_{n,j}$}\}
	\ee
	and
	\be 
	c_{\ell, R}=\frac{1}{\#(V(\Gamma_R))}\#\{\textrm{components $\Gamma_{R,j}$ of $\Gamma_R=\sqcup\Gamma_{R, j}$ such that $x_\ell$ belongs to the spectrum of $\Gamma_{R,j}$}\}.
	\ee
	In the above definitions of the coefficients $c_{\ell, n}$ and $c_{\ell, R}$, the components should be counted ``with multiplicity": if $x_\ell$ belongs to the spectrum of $\Gamma_{n, j}$ (respectively $\Gamma_{R, j}$) with some multiplicity $m(x_\ell)$, then the component $\Gamma_{n, j}$ (respectively $\Gamma_{R, j}$) is counted $m(x_\ell)$ times.
	
	We will need the following Lemma.
	We will need the following Lemma.
	\begin{lem}\label{lemg1}For every $\delta>0$ there exists $L>0$ such that 
	\be \mathbb{E}\sum_{\ell\geq L}c_{\ell,n}<\frac{\delta}{4}\quad \textrm{and}\quad \mathbb{E}\sum_{\ell\geq L}c_{\ell,R}<\frac{\delta}{4}.\ee
	\end{lem}
	\begin{proof}Let $A\subset\{x_\ell\}_{\ell\in \mathbb{N}} $ be any subset. Then:
	\be\sum_{\ell}\mathbb{E}c_{\ell, n}=\mathbb{E}\sum_{\ell}c_{\ell,n}\leq 1\quad \textrm{and}\quad \sum_{\ell}\mathbb{E}c_{\ell, R}=\mathbb{E}\sum_{\ell}c_{\ell,R}\leq 1.
	\ee
	In particular the two series $\mathbb{E}\sum_{\ell}c_{\ell,n}$ and $\mathbb{E}\sum_{\ell}c_{\ell,R}$ converge and therefore the existence of such $L$ is clear (as the tails of the series must be arbitrarily small).
	\end{proof}

	\begin{cor}For every $\ell\in \mathbb{N}$ there exists constants $c_\ell\geq 0$ such that we have the following convergence of random variables:
	\be c_{\ell, n}\stackrel{L^1}{\longrightarrow} c_\ell\quad \textrm{and}\quad c_{\ell, R}\stackrel{L^1}{\longrightarrow} c_\ell.\ee
	The constants $c_\ell$ are positive if and only if $x_\ell$ belongs to the spectrum of a $\R^m$--geometric graph. Moreover the measure
	\be\label{eq:measure} \mu:=\sum_{\ell\in \mathbb{N}}c_\ell \delta_{x_\ell}\ee
	is a probability measure on $\R$ with support contained in $[0,2]$.
	\end{cor}
	\begin{proof}The convergence in $L^1$ of the random variales $c_{\ell, n}$ and $c_{\ell, R}$ follows from their description as ``number of components such that $x_\ell$ belongs to the spectrum of the random geometric graph'': in fact for every $\ell\geq 0$ we can introduce the counting function:
	\be \mathcal{N}(\Gamma, \ell)=\#\{\textrm{components of  $\Gamma$ for which $x_\ell$ belongs to their spectrum}\}.\ee
	With this notation we have:
	\be c_{\ell, n}=\frac{\mathcal{N}(\Gamma_n,\ell) }{n}\quad \textrm{and}\quad c_{\ell, R}=\frac{\mathcal{N}(\Gamma_R,\ell) }{\textrm{vol}(B(0,R))}=\frac{\mathcal{N}(\Gamma_R,\ell) }{\#(V(\Gamma_R))}\cdot \frac{\#(V(\Gamma_R))}{\textrm{vol}(B(0,R))}.\ee
	The convergence of $c_{\ell, n}$ follows the exact same proof of Theorem \ref{teo4.1} and the convergence of $c_{\ell, R}$ proceeds as in Proposition \ref{prop2.1}. The measure $\mu$ is well defined and the fact that it is a probability measure follows from the same proof as \cite[Proposition 6.4]{Antonio}, using the above Lemma \ref{lemg1} as a substitute of \cite[Lemma 6.3]{Antonio}.
	\end{proof}
	\begin{theorem}\label{teomu}
		Let $\mu\in \mathcal{M}^1([0,2])$ be the measure defined in \eqref{eq:measure}. Then
		\begin{equation*}
		\mu_n\overset{*}{\underset{n\rightarrow\infty}{\rightharpoonup}}\mu\qquad \text{and}\qquad
		\mu_R\overset{*}{\underset{R\rightarrow\infty}{\rightharpoonup}}\mu.
		\end{equation*}
		i.e. $\forall f\in C^0([0,2],\mathbb{R})$ we have $\mathbb{E}\int f\diff \mu_n\rightarrow\int f\diff\mu$ and $\mathbb{E}\int f\diff \mu_R\rightarrow\int f\diff\mu$
	
	\end{theorem}
	\begin{proof}
	The proof of the statement for $\mu_n$ and $\mu_R$ is the same: we do it for $\mu_R$. Let $f\in C^{0}([0,2], \R)$, and fix $\varepsilon>0$. Apply Lemma \ref{lemg1} with the choice of $\delta=\varepsilon/\sup|f|$, and get the corresponding set $F=\{\ell_1, \ldots, \ell_a\}\subset \mathbb{N}$. Also, from the convergence of the series $\sum c_\ell$ we get the existence of a finite set $F'$ such that $\sum_{\ell\notin F}c_\ell<\varepsilon/\sup|f|$. Define $F''=F\cup F''$ (this is still a finite set) and:
	\begin{align} 
	\left|\mathbb{E}\int_{[0,2]}f d\mu_R-\int_{[0,2]}fd\mu\right|=&\left| \mathbb{E}\sum_\ell c_{\ell,R}f(x_\ell)-\sum_{\ell}c_{\ell}f(x_\ell)\right|\\
	=&\left|\mathbb{E}\sum_{\ell\in F''} c_{\ell,R}f(x_\ell)+\mathbb{E}\sum_{\ell\notin F''} c_{\ell,R}f(x_\ell)-\sum_{\ell\in F''} c_{\ell}f(x_\ell)-\sum_{\ell\notin F''} c_{\ell}f(x_\ell)\right|\\
	\leq &\left|\mathbb{E}\sum_{\ell\in F''} c_{\ell,R}f(x_\ell)-\sum_{\ell\in F''} c_{\ell}f(x_\ell)\right|\\
	&+\left|\mathbb{E}\sum_{\ell\notin F''} c_{\ell,R}f(x_\ell)\right|+\left|\sum_{\ell\notin F''} c_{\ell}f(x_\ell)\right|\\
	\leq &\left|\sum_{\ell\in F''} \mathbb{E}(c_{\ell,R}-c_\ell)f(x_\ell)\right|+\sup|f|\cdot\left| \sum_{\ell\notin F''} \mathbb{E}c_{\ell,R}\right|+\sup|f|\cdot\left| \sum_{\ell\notin F''} c_{\ell}\right|\\
	\leq &\sup|f|\sum_{\ell\in F''} \mathbb{E}|c_{\ell,R}-c_\ell|+\sup|f|\frac{\delta}{4}+\sup|f|\frac{\delta}{4}\\
	\label{eq:epsilon}\leq&\mathbb{E}|c_{\ell,R}-c_\ell|+\frac{\varepsilon}{2}\longrightarrow \frac{\varepsilon}{2}\quad \textrm{as $R\to \infty$},
	\end{align}
where in the last line we have used the $L^1$ convergence of $c_{\ell, R}\to c_\ell$. Since this is true for every $\varepsilon>0$, it follows that:
\be \lim_{R\to \infty}\mathbb{E}\int_{[0,2]}fd\mu_R=\int_{[0,2]}d\mu.\ee
	\end{proof}
	\begin{prop}\label{prop:last}
	The measure $\mu$ appearing in Theorem \ref{teomu} has the following properties:
		\begin{enumerate}
			\item $\mu$ is not absolutely continuous with respect to Lebesgue;
			\item $\lim_{n\rightarrow\infty}\frac{1}{n}\cdot\mathbb{E}\#\bigl(\text{eigenvalues of $\Gamma_n$ in } [a,b]\bigr)=\mu([a,b])>0$;
			\item $\mu(\{0\})=\beta$.
		\end{enumerate}
	\end{prop}
	
	\begin{proof}We start with point (3): we have that $\mu(\{0\})=c_{\ell_0}$ where $0=x_{\ell_0}$ and $c_{\ell_0}$ is the $L^1$-limit of $c_{\ell_0, R}$, which is the random variable:
	\be c_{\ell_0, R}=\#\{\textrm{components of $\Gamma_R$ containing $0$ in their spectrum}\}=b_0(\Gamma_R),\ee
	and therefore (3) follows from the definition of $\beta$.
	
	Point (1) also follows immediately, since $\beta_0>0$ and $\mu$ charges positively sets of Lebesgue measure zero (hence it cannot be absolutely continuous with respect to Lebesghe measure). 
	
	For the proof of point (2) we argue exactly as in the proof of Theorem \ref{teomu}, by replacing $f$ with $\chi_{[a,b]}$ (the characteristic function of the interval $[a,b]$) and observing that the only property of $f$ that we have used is its boundedness.
	\end{proof}
		\bibliographystyle{alpha}
	\bibliography{rgg20.03.08}	
\end{document}